\documentclass[a4paper]{article}
\usepackage{amsmath,amssymb,a4wide}
\usepackage{color,graphicx,epstopdf}
\usepackage[a4paper]{geometry}

\newcommand{\bu}{\boldsymbol u}
\newcommand{\bv}{\boldsymbol v}
\newcommand{\bw}{\boldsymbol w}

\newcommand{\bb}{\boldsymbol b}

\newcommand{\be}{\boldsymbol e}
\newcommand{\by}{\boldsymbol y}

\newcommand{\bs}{\boldsymbol s}
\newcommand{\bg}{\boldsymbol g}
\newcommand{\bff}{\boldsymbol f}

\newtheorem{Theorem}{Theorem}
\newtheorem{lema}{Lemma}
\newcounter{remark}
\def\theremark {\arabic{remark}}
\newenvironment{remark}{\refstepcounter{remark}\par\noindent{\bf Remark\ \theremark}\ }{\par}
\newtheorem{Proof}{Proof}

\newenvironment{proof}{\begin{Proof}\rm}{\hfill $\Box$ \end{Proof}}

\usepackage{hyperref}

\title{ Robust error bounds for the Navier-Stokes equations using implicit-explicit second order BDF method with variable steps
}
\author{Bosco
Garc\'{\i}a-Archilla\thanks{Departamento de Matem\'atica Aplicada
II, Universidad de Sevilla, Sevilla, Spain. Research is supported by
Spanish MCINYU under grants PGC2018-096265-B-I00 and PID2019-104141GB-I00 (bosco@esi.us.es)}
  \and Julia Novo\thanks{Departamento de
Matem\'aticas, Universidad Aut\'onoma de Madrid, Spain. Research is supported
by Spanish MINECO
under grants PID2019-104141GB-I00 and VA169P20  (julia.novo@uam.es)}}
\date{\today}
\begin{document}
\maketitle
\abstract{This paper studies fully discrete finite element approximations to the Navier-Stokes equations using inf-sup stable elements and grad-div stabilization. For
the time integration two implicit-explicit second order backward differentiation formulae (BDF2) schemes are applied. In both the laplacian is implicit while the nonlinear term is explicit, in the first one, and semi-implicit, in the second one. The grad-div
stabilization allow us to prove error bounds in which the constants are independent of inverse powers of the viscosity. Error bounds of order $r$ in space
are obtained for the $L^2$ error of the velocity using piecewise polynomials of degree $r$ to approximate the velocity together with second order bounds 
in time, both for fixed time step methods and for methods with variable time steps. A CFL-type condition is needed  for the method
in which the nonlinear term is explicit relating  time step and spatial mesh sizes parameters. 
}
\bigskip

\section{Introduction}
The numerical simulation of the Navier-Stokes equations is still a challenge in which there are several questions that deserve some research. In the
present paper we try to contribute to the following two aspects. On the one hand, we consider the possibility of optimize the cost of
the temporal integration. On the other, we consider spatial approximations that allow to prove error bounds with constants independent
of inverse powers of the viscosity. Concerning the first question, following the recent reference \cite{imex_schneier}, we consider
implicit-explicit methods with explicit or semi-implicit treatment of the nonlinear term, which reduces the cost of every single time step. Also, we
include the error analysis of the variable step case. As stated in \cite{john_var}: ``Adaptive time stepping is an important tool in Computational Fluid Dynamics for controlling the accuracy of simulations and for enhancing their efficiency". Concerning the second question, we add grad-div stabilization
 in the spatial discretization of the method. As in \cite{nos_grad_div}, adding this stabilization we are able to get bounds with constants independent of inverse powers of the viscosity. This is one of the methods considered in \cite{cmame_review} in which methods with
 robust bounds (error bounds with constants independent of the Reynolds number are studied). Methods for whom robust estimates can be derived enable stable flow simulations for small viscosity coefficients on comparatively coarse grids.

Let $\Omega \subset {\mathbb R}^d$, $d \in \{2,3\}$, be a bounded domain with
polyhedral  and Lipschitz boundary $\partial \Omega$. The incompressible
Navier--Stokes equations model the conservation of linear momentum and the
conservation of mass (continuity equation) by
\begin{align}
\label{NS} \partial_t\bu -\nu \Delta \bu + (\bu\cdot\nabla)\bu + \nabla p &= \bff &&\text{in }\ (0,T)\times\Omega,\nonumber\\
\nabla \cdot \bu &=0&&\text{in }\ (0,T)\times\Omega,\\
\bu(0, \cdot) &= \bu_0(\cdot)&&\text{in }\ \Omega,\nonumber
\end{align}
where $\bu$ is the velocity field, $p$ the kinematic pressure, $\nu>0$ the kinematic viscosity coefficient,
$\bu_0$ a given initial velocity, and $\bff$ represents the external body accelerations acting
on the fluid. The Navier--Stokes equations \eqref{NS} are equipped with homogeneous
Dirichlet boundary conditions $\bu = \boldsymbol 0$ on $\partial \Omega$.

In \cite{nos_grad_div} a stabilized method using only grad-div stabilization to approach the evolutionary Navier-Stokes equations is considered
and analyzed. Error bounds with constants independent of inverse powers of the viscosity are proved. Using piecewise polynomials 
of degree $r$ for the velocity bounds of order $r$ are obtained for the $L^2$ error of the velocity. Both the continuous-in-time case and the fully discrete scheme with the fully implicit backward Euler method as time integrator are analyzed.

Recently, in \cite{imex_schneier}, implicit-explicit (IMEX) variable stepsize methods for the evolutionary Navier-Stokes equations are analyzed. In these
algorithms the nonlinear term is treated fully explicitly while the remaining terms are treated implicitly. Stability and convergence is proved
for a variable stepsize first order method. However, the error bounds in \cite{imex_schneier} depend strongly on inverse powers of $\nu$. In particular, 
a CFL-type condition is needed for the time step depending on $\nu^{-1}$.  The authors state that the analysis including grad-div stabilization remains
being an open problem.

 In the present paper, as in \cite{nos_grad_div}, we add grad-did stabilization to the mixed finite element approximation
to be able to get bounds independent on $\nu^{-1}$. We analyze a second order BDF2 method in time, with explicit treatment of the nonlinear term.
We consider both the case of fixed time step and the case of variable time step.
A CFL-type condition is needed in our error analysis. This condition is stronger than the usual one in which $(\Delta t) h^{-1}$ has to be bounded ($h$
being the mesh size) because we required $(\Delta t) h^{-2}$ to be bounded (a similar CFL condition is obtained for the variable step case). However, we improve the condition in \cite{imex_schneier} in which
the quantity that has to be bounded is $(\Delta t) h^{-1}\nu^{-1}$. Then, in principle, the method we consider could be applied in case  of high Reynolds numbers without 
the need of reducing the size of the time step. Due to the stronger CFL condition of the method with  explicit treatment
of the nonlinear term, depending on the problem, it could be worth to use instead a semi-implicit
form for the nonlinear term. In this paper, we also include the error analysis of such a method in which the
nonlinear term is semi-implicit. For this method the error analysis  is obtained with the
same tools as the previous analysis but it is much simpler. Also,  no CFL condition is needed. The same kind of bounds as for the previous method
are proved in which the constants are also independent on inverse powers of the viscosity.

Concerning the BDF2 method with variable time step, we did not found in the literature any reference with the
error analysis for the Navier-Stokes equations. In \cite{becker} a second order backward difference method with variable steps is
analyzed for a linear parabolic problem. The error analysis for semilinear parabolic problems can be found in \cite{emmrich}.
The convergence of the variable two-step BDF time discretisation of nonlinear evolution problems governed by a monotone potential operator is
studied in \cite{emmrich2} while the analysis for the Cahn-Hilliard equation appears in \cite{chinos_CH}.

The outline of the paper is as follows. In Section 2 we introduce some notation and state some preliminaries. In Section 3 we carry out the error
analysis of the fully discrete methods, both for the fixed step and the variable time step cases. Some technical lemmas are stated and proved along the section. Although the error analysis of the fixed step can be obtained as a particular case of the variable step  we decided to start with the analysis of the simpler case since most of the ideas are common
and in this way, in our opinion, the skeleton of the analysis can be better understood. We prove  optimal convergence of second order in time for the
fully discrete methods (both for the fixed and variable time step size cases and both for explicit and semi-implicit treatment
of the nonlinear term) and
the same rate of convergence in space as in \cite{nos_grad_div}. In Section 4 we show some numerical experiments while 
some conclusions are stated in the last section.

\section{Preliminaries and notation}
Throughout the paper we will denote by $W^{s,p}(D)$ the Sobolev space of real valued functions defined on the domain $D\subset\mathbb{R}^d$ with distributional derivatives of order up to $s$ in $L^p(D)$, endowed with the usual norm which is denoted by $\|\cdot\|_{W^{s,p}(D)}$. If $s$ is not a positive integer, $W^{s,p}(D)$ is defined by interpolation \cite{Adams}. If $s=0$ we understand that $W^{0,p}(D)=L^p(D)$. As it is standard, $W^{s,p}(D)^d$ will be endowed with the product norm that, if no confusion can arise, will be denoted again by $\|\cdot\|_{W^{s,p}(D)}$. We will distinguish the case  $p=2$ using $H^s(D)$ to denote the space $W^{s,2}(D)$. We will make use of the space $H_0^1(D)$, the closure in $H^1(D)$ of the set of infinitely differentiable functions with compact support
in $D$.  For simplicity,
we use $\|\cdot\|_s$ (resp. $|\cdot |_s$) to denote the norm (resp. seminorm) both in $H^s(\Omega)$ or $H^s(\Omega)^d$. The exact meaning will be clear by the context. The inner product of $L^2(\Omega)$ or $L^2(\Omega)^d$ will be denoted by $(\cdot,\cdot)$ and the corresponding norm by $\|\cdot\|_0$. The norm of the space of essentially bounded functions $L^\infty(\Omega)$ will be denoted by $\|\cdot\|_\infty$. For vector valued function we will use the same conventions as before.
We represent by $\|\cdot\|_{-1}$  the norm of the dual space of $H^1_0(\Omega)$ which is denoted by $H^{-1}(\Omega)$. As usual, we always identify $L^2(\Omega)$
with its dual so we have $H^1_0(\Omega)\subset L^2(\Omega)\subset H^{-1}(\Omega) $ with compact injection.

Using the function spaces
$
V=H_0^1(\Omega)^d$, and $$ Q=L_0^2(\Omega)=\left\{q\in L^2(\Omega):
(q,1)=0\right\},
$$
the weak formulation of problem (\ref{NS}) is:
Find $(\bu,p):(0,T]\rightarrow  V\times Q$ such that for all $(\bv,q)\in V\times Q$,
\begin{equation}\label{eq:NSweak}
(\partial_t\bu,\bv)+\nu (\nabla \bu,\nabla \bv)+((\bu\cdot \nabla) \bu,\bv)-(\nabla \cdot \bv,p)
+(\nabla \cdot \bu,q)=(\boldsymbol f,\bv).
\end{equation}

The Hilbert space
$$
H^{\rm div}=\{ \bu \in L^{2}(\Omega)^d \ \mid \ \nabla \cdot \bu=0, \,
\bu\cdot \mathbf n|_{\partial \Omega} =0 \}$$
will be endowed with the inner product of $L^{2}(\Omega)^{d}$ and the space
$$V^{\rm div}=\{ \bu \in V \ \mid \ \nabla \cdot \bu=0 \}$$
with the inner product of $V$.

The following Sobolev's embedding \cite{Adams} will be used in the analysis: For  $1\le p<d/s$
let $q$ be such that $\frac{1}{q}
= \frac{1}{p}-\frac{s}{d}$. There exists a positive constant $C$, independent of $s$, such
that
\begin{equation}\label{sob1}
\|v\|_{L^{q'}(\Omega)} \le C \| v\|_{W^{s,p}(\Omega)}, \qquad
\frac{1}{q'}
\ge \frac{1}{q}, \quad v \in
W^{s,p}(\Omega).
\end{equation}
If $p>d/s$ the above relation is valid for $q'=\infty$. A similar embedding inequality holds trivially for vector valued functions.

Let $V_h\subset V$ and $Q_h\subset Q$ be two families of finite element
spaces composed of piecewise polynomials of
degrees at most $k$ and $l$, respectively, that correspond to a family of partitions $\mathcal T_h$ of
$\Omega$ into mesh cells with maximal diameter $h$.
In the following, we consider pairs of finite element spaces that satisfy the
inf-sup condition
\begin{equation}\label{LBB}
\inf_{q_h\in Q_h}\sup_{\bv_h\in V_h}\frac{(\nabla \cdot \bv_h,q_h)}{\|\nabla \bv_h\|_0\|q_h\|_0}\ge \beta_0,
\end{equation}
with $\beta_0$ a constant independent of the mesh size $h$.

It will be assumed that
the meshes are quasi-uniform and that  the following inverse
inequality holds for each $v_{h} \in V_{h}$, see e.g., \cite[Theorem 3.2.6]{Cia78},
\begin{equation}
\label{inv} \| \bv_{h} \|_{W^{m,p}(K)} \leq C_{\mathrm{inv}}
h_K^{n-m-d\left(\frac{1}{q}-\frac{1}{p}\right)}
\|\bv_{h}\|_{W^{n,q}(K)},
\end{equation}
where $0\leq n \leq m \leq 1$, $1\leq q \leq p \leq \infty$ and $h_K$
is the size (diameter) of the mesh cell $K \in \mathcal T_h$.

The space of discrete divergence-free functions is denoted by
$$
V_h^{\rm div}=\left\{\bv_h\in V_h\ \mid\ (\nabla\cdot \bv_h,q_h)=0\quad \forall q_h\in Q_h \right\}.
$$


Denoting by $\pi_h$ the $L^2(\Omega)$ projection of the pressure $p$
onto $Q_h$, we have that for $0\le m\le 1$
\begin{equation}\label{eq:pressurel2}
\|p-\pi_h\|_m\le C h^{l+1-m}\|p\|_{l+1} \quad \forall p\in H^{l+1}(\Omega).
\end{equation}
In the error analysis, the Poincar\'{e}--Friedrichs inequality
\begin{equation}\label{eq:poin}
\|\bv\|_{0} \leq
C\left|\Omega\right|^{1/d}\|\nabla \bv\|_{0},\quad\forall \bv\in H_0^1(\Omega)^d,
\end{equation}
will be used.


In the analysis, the Stokes problem
\begin{align}\label{eq:stokes_str}
-\nu \Delta \bu+\nabla p&={\bg}\quad\mbox{in }\Omega, \nonumber\\
\bu&=\boldsymbol 0\quad\mbox{on } \partial \Omega,\\
\nabla \cdot \bu&=0\quad\mbox{in } \Omega,\nonumber
\end{align}
will be considered. Let us denote by
$(\bu_h,p_h)\in V_h \times Q_h$  the mixed finite element approximation to \eqref{eq:stokes_str},
given by
\begin{equation}\label{eq:gal_stokes}
\begin{split}
\nu(\nabla \bu_h,\nabla \bv_h)-(\nabla \cdot \bv_h, p_h)&=({\bg},\bv_h)\quad \forall \bv_h\in V_h\\
(\nabla \cdot \bu_h,q_h)&=0\qquad\quad\,\,\,  \forall q_h\in Q_h.
\end{split}
\end{equation}
Following \cite{girrav}, one gets the estimates
\begin{align}\label{eq:cota_stokes_v0}
h^{-1}\|\bu-\bu_h\|_0+\|\bu-\bu_h\|_1&\le  C \left(\inf_{\bv_h\in V_h}\|\bu-\bv_h\|_1+\nu^{-1}\inf_{q_h\in Q_h}\|p-q_h\|_0\right),\\
\|p-p_h\|_0&\le C\left( \nu \inf_{\bv_h\in V_h}\|\bu-\bv_h\|_1+
\inf_{q_h\in Q_h}\|p-q_h\|_0\right).\label{eq:cota_stokes_pre}\end{align}
It can be observed that the error bounds for the velocity depend on negative powers of $\nu$.

For the analysis, it will be advantageous to use a projection of $(\bu,p)$ into $V_h \times Q_h$
with uniform in $\nu$, optimal, bounds for the velocity. In \cite{nos_oseen} a projection with this property
 was introduced.
Let  $(\bu,p)$ be the solution of the Navier--Stokes equations \eqref{NS} with $\bu\in V\cap H^{k+1}(\Omega)^d$,
$p \in Q\cap H^{k}(\Omega)$,
$k\ge 1$,
and observe that $(\bu,0)$ is the solution of  the Stokes problem \eqref{eq:stokes_str} with
 right-hand side
\begin{equation}\label{eq:stokes_rhs_g}
{\bg}={\bff}-\partial_t \bu-(\bu\cdot \nabla)\bu-\nabla p.
\end{equation}
Denoting the corresponding
Galerkin approximation in $V_h\times Q_h$ by $({\bs_h},l_h)$, one obtains from
\eqref{eq:cota_stokes_v0}-\eqref{eq:cota_stokes_pre}
\begin{align}\label{eq:cotanewpro}
\|\bu-{\bs}_h\|_0+h\|\bu-{\bs}_h\|_1&\le C h^{k+1}\|\bu \|_{k+1},\\
\label{eq:cotanewpropre}
\|l_h\|_0&\le C \nu h^{k}\|\bu\|_{k+1},
\end{align}
where the constant $C$ does not depend on $\nu$.

\begin{remark}\label{rem:estimate_partial_t}
Assuming the necessary smoothness in time and considering \eqref{eq:stokes_str} with
$$
{\bg} = \partial_t\left({\bff}-\partial_t \bu-(\bu\cdot \nabla)\bu-\nabla p\right),
$$
one can derive an error bound of the form \eqref{eq:cotanewpro}
also for $\partial_t (\bu-{\bs}_h)$. One can proceed similarly for higher order
derivatives in time.
\end{remark}

Following \cite{chenSiam},
one can also obtain the following bounds for ${\bs}_h$
\begin{align}
\|\bu-{\bs}_h\|_\infty&\le C_\infty h\left(\log(\left|\Omega\right|^{1/d}/h)\right)^{\bar r}\|\nabla \bu\|_\infty\label{cotainfty0},\\
\|\nabla (\bu-{\bs}_h)\|_\infty&\le C_\infty\|\nabla \bu\|_\infty \label{cotainfty1},
\end{align}
where $C_\infty$ does not depend on $\nu$ and $\bar r=0$ for linear elements and $\bar r=1$ otherwise.

\section{The IMEX and semi-implicit BDF2 methods in time with grad-div stabilization in space}
In the sequel we well call IMEX to the method with explicit treatment of the nonlinear term and semi-implicit to
the method with semi-implicit nonlinear term.

In this section we state and analyze a fixed step and a variable step versions of both an IMEX and a semi-implicit
BDF2 methods with grad-div stabilization in space. We first introduce the grad-div stabilization
and a couple of Gronwall lemmas that will be applied for the error analysis of the methods. Then,
we introduce the fully discrete schemes. Some lemmas are then proved that will be used both
for the fixed and variable step cases analyzed in next two subsections. In Section \ref{se:fixed} we carry out the error analysis for the case in which the step size is fixed. The variable step size case
is analyzed in Section \ref{se:var}. The ideas in both sections are similar, the second one having more technical difficulties.
For this reason, we think it is convenient to show first the analysis for the fixed step size, although this
case can be obviously obtained as a particular case of the error analysis of Section \ref{se:var}.

The spatial discretization that will be studied for the approximation of the solution of
the Navier--Stokes  equations
(\ref{NS}) is obtained by adding to the Galerkin equations a control of the divergence
constraint (grad-div stabilization). More precisely,
the following grad-div method will be considered: Find $(\bu_h,p_h):(0,T]\rightarrow V_h\times Q_h$ such that for all $(\bv_h,q_h)\in V_h\times Q_h$ one has
\begin{equation}\label{eq:gal}
\begin{split}
(\partial_t\bu_h,\bv_h)+\nu(\nabla \bu_h,\nabla \bv_h)&+b(\bu_h,\bu_h,\bv_h)-(p_h,\nabla \cdot \bv_h)\\
&+(\nabla\cdot \bu_h,q_h)+\mu(\nabla \cdot \bu_h,\nabla \cdot \bv_h) =(\boldsymbol f,\bv_h),
\end{split}
\end{equation}
with $\bu_h(0)$ given.
Here, and in the rest of the paper,
$$
b(\bu,\bv,\bw)=(B(\bu,\bv),\bw)\quad \forall\bu,\bv,\bw\in H_0^1(\Omega)^d,
$$
where,
$$
B(\bu,\bv)=(\bu\cdot \nabla )\bv+\frac{1}{2}(\nabla  \cdot\bu)\bv\quad \forall \bu,\bv\in H_0^1(\Omega)^d
$$
Notice the well-known property
\begin{equation}
\label{eq:skew}
b(\bu,\bv,\bw)= - b(\bu,\bw,\bv)\quad \forall\bu,\bv,\bw\in V,
\end{equation}
such that, in particular, $b(\bu,\bw,\bw) = 0$ $\forall \bu, \bw\in V $.
We will use the following discrete Gronwall inequality whose proof can be found in \cite{hey-ran4}.
\begin{lema}\label{gronwall}
Let $k,B,a_j,b_j,c_j,\gamma_j$ be nonnegative numbers such that
$$
a_n+k\sum_{j=0}^n b_j\le k\sum_{j=0}^n \gamma_j a_j+k\sum_{j=0}^n c_j+B,\quad for \quad n\ge 0.
$$
Suppose that $k\gamma_j<1$, for all j, and set $\sigma_j=(1-k\gamma_j)^{-1}$. Then
$$
a_n+k\sum_{j=0}^nb_j\le \exp\left(k\sum_{j=0}^n\sigma_j\gamma_j\right)\left\{k\sum_{j=0}^nc_j+B\right\}, \quad for\quad  n\ge 0.
$$
\end{lema}
We will also apply the following Gronwall lemma.
\begin{lema}\label{gronwall2}
Let $B,a_j,b_j,c_j,\gamma_j,\delta_j$ be nonnegative numbers such that
$$
a_n+\sum_{j=0}^n b_j\le \gamma_n a_n +\sum_{j=0}^{n-1} (\gamma_j+\delta_j) a_j+\sum_{j=0}^n c_j+B,\quad for \quad n\ge 0.
$$
Suppose that $\gamma_j<1$, for all j, and set $\sigma_j=(1-\gamma_j)^{-1}$. Then
$$
a_n+\sum_{j=0}^nb_j\le \exp\left(\sigma_n \gamma_n+\sum_{j=0}^{n-1}(\sigma_j\gamma_j+\delta_j)\right)\left\{\sum_{j=0}^nc_j+B\right\}, \quad for\quad  n\ge 0.
$$
\end{lema}
\begin{proof}
We follow the proof of~\cite[Lemma~5.1]{hey-ran4}. Let $d_n$ be defined by the relation
$$
a_n+\sum_{j=0}^n b_j+d_n= \gamma_n a_n +\sum_{j=0}^{n-1} (\gamma_j+\delta_j) a_j+\sum_{j=0}^n c_j+B,\quad for \quad n\ge 0,
$$
and let us denote by~$S_n$ the left hand side above. Notice then 
$$
S_n-S_{n-1} = \gamma_n a_n + \delta_{n-1} a_{n-1} + c_n \le \gamma_n S_n + \delta_{n-1} S_{n-1}+c_n,\qquad n\ge 1.
$$
For $n=0$, we have
$$
S_0 \le \left(1+\gamma_0\right)^{-1} (c_0+B) = \left(1 + \sigma_0\gamma_0\right) (c_0+B)
\le \exp(\sigma_0\gamma_0)(c_0+B).
$$
Assume that
$$
S_k \le \exp\biggl(\sigma_k \gamma_k + \sum_{j=0}^{k-1} (\sigma_j \gamma_j+ \delta_j)\biggr) \biggl(
\sum_{j=0}^{k} c_j +B\biggr), \qquad k=0,\ldots,n.
$$
and let us see that the inequality also holds for $k=n+1$, which will finish the proof. We have
\begin{align*} 
S_{n+1} &\le (1-\gamma_{n+1})^{-1} \left((1+\delta_n)S_n + c_{n+1}\right)
\le (1+\sigma_{n+1}\gamma_{n+1})\left(\exp(\delta_n) S_n+c_{n+1}\right)\\
& \le \exp\left(\sigma_{n+1}\gamma_{n+1}\right)\biggl( \exp\biggl(\sum_{k=0}^n(\sigma_k\gamma_k+\delta_k)\biggr)
\biggl(\sum_{k=0}^n c_k+B\biggr) + c_{n+1}\biggr).
\end{align*}
\end{proof}

Now, we consider a fully discrete approximation to the solution of \eqref{eq:gal}. 
Let us consider time levels $0=t_0<t_1<\ldots < t_N=T$, and let us denote
$$
\Delta t_n =t_{n+1}-t_n, \qquad n=0,1\ldots,N-1,
$$
and
$$
\omega_n=\Delta t_n /\Delta t_{n-1},\qquad n=1,\ldots,N-1.
$$
In the sequel we will assume that for $\gamma\in(0,1)$ and $\Gamma>1$, we have
\begin{equation}
\label{eq:cond_omega}
\gamma \le \omega_n\le \Gamma,\qquad n=0,1,\ldots,N-1.
\end{equation}
We wil also use the following notation,
$$
D \bu^n=\bu^n-\bu^{n-1},\quad \hat \bu^n= \bu^n+\omega_nD\bu^n,
$$
and
$$
{\cal D}\bu^{n} = D\bu^{n} +\frac{\omega_{n-1}}{1+\omega_{n-1}}\left( \bu^{n} -\hat\bu^{n-1}\right).
$$
Observe that $\hat \bu^n$ is the extrapolated value at $t_{n+1}$ of the linear interpolant taking values $\bu^n$ and~$\bu^{n-1}$ at $t_n$ and~$t_{n-1}$, respectively. Notice also that for fixed step size, that is when $\Delta t_n=T/N$, 
$n=0,1,\ldots,N-1$, we have
$$
\hat \bu^n = 2\bu^n - \bu^{n-1}\quad \hbox{\rm and} \quad {\cal D}\bu^n =(D+\frac{1}{2}D^2)\bu^n=\frac{3}{2}\bu^n-2\bu^{n-1}+\frac{1}{2}\bu^{n-2}.
$$

Now, we consider the following implicit-explicit (IMEX) approximation to \eqref{eq:gal} based on the BDF2 method.
Given $\bu_h^0, \bu_h^1$ solve, for $n\ge 1$ and for all $(\bv_h,q_h)\in V_h\times Q_h$ 
\begin{eqnarray}\label{fully_discrete}
&&\left(\frac{1}{\Delta t_n}{\cal D}\bu_h^{n+1},\bv_h\right)+\nu(\nabla \bu_h^{n+1},\nabla \bv_h)+b(\hat \bu_h^n,\hat \bu_h^n,\bv_h)
-(p_h^{n+1},\nabla \cdot \bv_h)\quad\\
&&\quad\quad+(\nabla\cdot \bu_h^{n+1},q_h)+\mu(\nabla \cdot \bu_h^{n+1},\nabla \cdot \bv_h) =(\boldsymbol f^{n+1},\bv_h),\nonumber
\end{eqnarray} 
where, here and in the sequel, the notation $\boldsymbol f^{n+1}$ means $\boldsymbol f(t_{n+1})$.
We also consider
the semi-implicit method obtained when the third term on the left-hand side of ~(\ref{fully_discrete}) is replaced by
\begin{equation}
\label{eq:elCI0}
b(\hat \bu_h^n, \bu_h^{n+1},\bv_h)
\end{equation}
that is,
\begin{eqnarray}\label{semimp}
&&\left(\frac{1}{\Delta t_n}{\cal D}\bu_h^{n+1},\bv_h\right)+\nu(\nabla \bu_h^{n+1},\nabla \bv_h)+b(\hat \bu_h^n,\bu_h^{n+1},\bv_h)
-(p_h^{n+1},\nabla \cdot \bv_h)\quad\\
&&\quad\quad+(\nabla\cdot \bu_h^{n+1},q_h)+\mu(\nabla \cdot \bu_h^{n+1},\nabla \cdot \bv_h) =(\boldsymbol f^{n+1},\bv_h),\nonumber
\end{eqnarray}

Taking $\bv_h\in V_h^{\rm div}$ in \eqref{fully_discrete} we get
\begin{eqnarray}\label{fully_discrete2}
\left(\frac{1}{\Delta t_n}{\cal D}\bu_h^{n+1},\bv_h\right)+\nu(\nabla \bu_h^{n+1},\nabla \bv_h)+b(\hat \bu_h^n,\hat \bu_h^n,\bv_h)+\mu(\nabla \cdot \bu_h^{n+1},\nabla \cdot \bv_h) =(\boldsymbol f^{n+1},\bv_h).
\end{eqnarray} 
To get $\bu_h^1$ we 
apply the IMEX Euler method with the same structure as BDF2 (explicit non-linear term) for both the IMEX and the semi-implicit methods. In fact, formulae~\eqref{fully_discrete} and~\eqref{fully_discrete2} are valid for $n=0$ if we set
\begin{equation}
\label{eq:notation0}
\hat\bu^0=\bu^0,\quad {\cal D} \bu^1 =D\bu^1.
\end{equation}

We will compare $\bu_h^n$ with $\bs_h^n$. For simplicity we will assume $\bu_h^0=\bs_h^0$. 
It is easy to prove that for $n\ge 1$ and $\bv_h\in V_h^{\rm div}$
\begin{eqnarray}\label{fully_discrete2_sh}
&&\left(\frac{1}{\Delta t_n}{\cal D}\bs_h^{n+1},\bv_h\right)+\nu(\nabla \bs_h^{n+1},\nabla \bv_h)+b(\hat \bs_h^n,\hat \bs_h^n,\bv_h)
+\mu(\nabla \cdot \bs_h^{n+1},\nabla \cdot \bv_h) \quad\\
&&\quad=(\boldsymbol f^{n+1},\bv_h)+(p^{n+1}-\pi_h^{n+1},\nabla \cdot \bv_h)
+\mu(\nabla \cdot (\bs_h^{n+1}-\bu^{n+1}),\nabla \cdot \bv_h)\nonumber\\
&&\quad+\left(\frac{1}{\Delta t_n}{\cal D}\bs_h^{n+1}-\bu_t^{n+1},\bv_h\right)
+b(\hat \bs_h^n,\hat\bs_h^n,\bv_h)-b(\bu^{n+1},\bu^{n+1},\bv_h).
\nonumber
\end{eqnarray}
Then, denoting by 
$$
\be_h^n=\bs_h^n-\bu_h^n.
$$
and subtracting \eqref{fully_discrete2} from \eqref{fully_discrete2_sh} we get
\begin{eqnarray}\label{fully_error}
&&\left(\frac{1}{\Delta t_n}{\cal D}\be_h^{n+1},\bv_h\right)+\nu(\nabla \be_h^{n+1},\nabla \bv_h)
+b(\hat \bs_h^n,\hat \bs_h^n,\bv_h)-b(\hat\bu_h^n,\hat\bu_h^n,\bv_h)\quad\nonumber\\
&&\quad+\mu(\nabla \cdot \be_h^{n+1},\nabla \cdot \bv_h) =(\tau_1^{n+1}+\tau_3^{n+1},\bv_h)+(\tau_2^{n+1},\nabla \cdot \bv_h).
\end{eqnarray}
where
\begin{eqnarray}\label{lostrun}
\tau_1^{n+1}&=&\frac{1}{\Delta t_n}{\cal D}\bs_h^{n+1}-\bu_t^{n+1},\nonumber\\
\tau_2^{n+1}&=&(p^{n+1}-\pi_h^{n+1})+\mu\nabla \cdot(\bs_h^{n+1}-\bu^{n+1}),\\
(\tau_3^{n+1},\bv_h)&=&b(\hat \bs_h^n,\hat\bs_h^n,\bv_h)-b(\bu^{n+1},\bu^{n+1},\bv_h).\nonumber
\end{eqnarray} 
Similarly, for the semi-implicit method we get~\eqref{fully_discrete2} with the third and foruth terms replaced by
$b(\hat \bs_h^n, \bs_h^{n+1},\bv_h)-b(\hat\bu_h^n,\bu_h^{n+1},\bv_h)$, and $\tau_3$ replaced by $\tau_4$ defined by
\begin{equation}
\label{tau4}
(\tau_4^{n+1},\bv_h)=b(\hat \bs_h^n,\bs_h^{n+1},\bv_h)-b(\bu^{n+1},\bu^{n+1},\bv_h).
\end{equation}

In the following lemma we bound the difference of the nonlinear terms for the IMEX method.
\begin{lema}\label{le:nl}  Considering the fixed stepsize case $\Delta t_n=\Delta t=T/N$, the following bound holds
\begin{eqnarray}\label{eq:nonli1-0}
&&|b(\hat \bs_h^n,\hat \bs_h^n,\be_h^{n+1})-b(\hat\bu_h^n,\hat\bu_h^n,\be_h^{n+1})|\le
\left(\frac{\|\nabla \hat \bs_h^n\|_\infty}{2}+\frac{9}{2\mu}\|\hat \bs_h^n\|_\infty^2\right)\|\be_h^{n+1}\|_0^2
\nonumber\\
&&\quad+3\|\nabla \hat \bs_h^n\|_\infty\|\be_h^n\|_0^2+\frac{3}{2}\|\nabla \hat \bs_h^n\|_\infty\|\be_h^{n-1}\|_0^2+\frac{\mu}{3}\|\nabla \cdot\be_h^n\|_0^2+\frac{\mu}{6}\|\nabla \cdot\be_h^{n-1}\|_0^2\\
&&\quad+2\|\hat \bu_h^n\|_\infty c_{\rm inv} h^{-1}\|\be_h^{n+1}-\hat \be_h^n\|_0\|\be_h^{n+1}\|_0.
\nonumber
\end{eqnarray}
In view of the upper bound in \eqref{eq:cond_omega}, for the variable stepsize case the bound for the difference of the nonlinear terms is
\begin{eqnarray}\label{eq:nonli1-0_var}
&&|b(\hat \bs_h^n,\hat \bs_h^n,\be_h^{n+1})-b(\hat\bu_h^n,\hat\bu_h^n,\be_h^{n+1})|\le
\left(\frac{\|\nabla \hat \bs_h^n\|_\infty}{2}+\frac{9\alpha}{2\mu}\|\hat \bs_h^n\|_\infty^2\right)\|\be_h^{n+1}\|_0^2
\nonumber\\
&&\quad+\frac{(1+2\Gamma)(1+\Gamma)}{2}\|\nabla \hat \bs_h^n\|_\infty\|\be_h^n\|_0^2
+\frac{\Gamma(1+2\Gamma)}{2}\|\nabla \hat \bs_h^n\|_\infty\|\be_h^{n-1}\|_0^2
\nonumber\\
&&\quad+\frac{\mu(1+2\Gamma)(1+\Gamma)}{18\alpha}\|\nabla \cdot\be_h^n\|_0^2
\quad+\frac{\mu\Gamma(1+2\Gamma)}{18\alpha}\|\nabla \cdot\be_h^{n-1}\|_0^2\\
&&\quad+2\|\hat \bu_h^n\|_\infty c_{\rm inv} h^{-1}\|\be_h^{n+1}-\hat \be_h^n\|_0\|\be_h^{n+1}\|_0,
\nonumber
\end{eqnarray}
where $\alpha$ is a positive constant.
\end{lema}
\begin{proof}
We start by proving \eqref{eq:nonli1-0}.
Adding $\pm b(\hat \bu_h^n,\hat \bs_h^n,\be_h^{n+1})$, and noticing that $b(\hat \bu_h^n,\be_h^{n+1},\be_h^{n+1})=0$,
we have
\begin{equation}
\label{eq:nl0}
|b(\hat \bs_h^n,\hat \bs_h^n,\be_h^{n+1})-b(\hat\bu_h^n,\hat\bu_h^n,\be_h^{n+1})|
=|b(\hat \be_h^n,\hat \bs_h^n,\be_h^{n+1})
-b(\hat \bu_h^n,\be_h^{n+1}-\hat \be_h^n,\be_h^{n+1})|.
\end{equation}
For the first term on the right-hand side above we write
\begin{eqnarray}\label{eq:hacefal}
|b(\hat \be_h^n,\hat \bs_h^n,\be_h^{n+1})|&\le&
\|\hat \be_h^n\|_0\|\nabla \hat \bs_h^n\|_\infty\|\be_h^{n+1}\|_0
+\|\hat \bs_h\|_\infty\|\nabla \cdot \hat \be_h^n\|_0\|\be_h^{n+1}\|_0\\
&\le&\left(\frac{\|\nabla \hat \bs_h^n\|_\infty}{2}+\frac{9}{2\mu}\|\hat\bs_h^n\|_\infty^2\right)\|\be_h^{n+1}\|_0^2
+\|\nabla \hat \bs_h^n\|_\infty\frac{1}{2}\|\hat \be_h^n\|_0^2+
\frac{\mu}{18}\|\nabla \cdot \hat \be_h^n\|_0^2.\nonumber
\end{eqnarray}
Let us now observe that, 
$$
\|\hat \be_h^n\|_0^2\le 6\|\be_h^n\|_0^2+3\|\be_h^{n-1}\|_0^2,
$$
and a similar expression for~$\|\nabla\cdot \hat \be_h^n\|_0^2$, so that,
\begin{eqnarray*}
|b(\hat \be_h^n,\hat \bs_h^n,\be_h^{n+1})|&\le&
\left(\frac{\|\nabla \hat \bs_h^n\|_\infty}{2}+\frac{9}{2\mu}\|\hat\bs_h^n\|_\infty^2\right)\|\be_h^{n+1}\|_0^2
{}+\frac{\mu}{3}\|\nabla \cdot\be_h^n\|_0^2+\frac{\mu}{6}\|\nabla \cdot\be_h^{n-1}\|_0^2\\
&&+3\|\nabla \hat \bs_h^n\|_\infty\|\be_h^n\|_0^2+\frac{3}{2}\|\nabla \hat \bs_h^n\|_\infty\|\be_h^{n-1}\|_0^2.
\end{eqnarray*}
Finally, from the definition of the nonlinear term and using inverse inequality \eqref{inv} it is easy to check that
$$
|b(\hat \bu_h^n,\be_h^{n+1}-\hat \be_h^n,\be_h^{n+1})| \le 2\|\hat \bu_h^n\|_\infty c_{\rm inv} h^{-1}\|\be_h^{n+1}-\hat \be_h^n\|_0\|\be_h^{n+1}\|_0.
$$
The proof of \eqref{eq:nonli1-0_var} can be obtained 
arguing in the same way with only two differences. The first one is that we include a parameter
$\alpha$ to bound the second term on the first line of the right-hand side of \eqref{eq:hacefal}. The
second one is that we  apply \eqref{eq:cond_omega} to bound both $\|\hat \be_h^n\|_0^2$ and $\| \nabla \cdot \hat \be_h^n\|_0^2$.
\end{proof}
For the semi-implicit method we have the following result.

\begin{lema}\label{le:nl-si}   The following bound holds in the fixed stepsize case $\Delta t_n=\Delta t=T/N$,
\begin{eqnarray}\label{eq:nonli1-0-si}
&&|b(\hat \bs_h^n, \bs_h^{n+1},\be_h^{n+1})-b(\hat\bu_h^n,\bu_h^{n+1},\be_h^{n+1})|\le
\left(\frac{\|\nabla \bs_h^{n+1}\|_\infty}{2}+\frac{9}{2\mu}\|\bs_h^{n+1}\|_\infty^2\right)\|\be_h^{n+1}\|_0^2
\nonumber\\
&&\quad+3\|\nabla\bs_h^{n+1}\|_\infty\|\be_h^n\|_0^2+\frac{3}{2}\|\nabla\bs_h^{n+1}\|_\infty\|\be_h^{n-1}\|_0^2+\frac{\mu}{3}\|\nabla \cdot\be_h^n\|_0^2+\frac{\mu}{6}\|\nabla \cdot\be_h^{n-1}\|_0^2,
\end{eqnarray}
and, in the variable stepsize case, for $\Gamma$ the constant in~\eqref{eq:cond_omega} and~$\alpha$ any positive constant,
the following bound holds:
\begin{eqnarray}\label{eq:nonli1-0_var-si}
&&|b(\hat \bs_h^n \bs_h^{n+1},\be_h^{n+1})-b(\hat\bu_h^n,\bu_h^{n+1},\be_h^{n+1})|\le
\left(\frac{\|\nabla \bs_h^{n+1}\|_\infty}{2}+\frac{9\alpha}{2\mu}\|\bs_h^{n+1}\|_\infty^2\right)\|\be_h^{n+1}\|_0^2
\nonumber\\
&&\quad+\frac{(1+2\Gamma)(1+\Gamma)}{2}\|\nabla \bs_h^{n+1}\|_\infty\|\be_h^n\|_0^2
+\frac{\Gamma(1+2\Gamma)}{2}\|\nabla\bs_h^{n+1}\|_\infty\|\be_h^{n-1}\|_0^2
\nonumber\\
&&\quad+\frac{\mu(1+2\Gamma)(1+\Gamma)}{18\alpha}\|\nabla \cdot\be_h^n\|_0^2
\quad+\frac{\mu\Gamma(1+2\Gamma)}{18\alpha}\|\nabla \cdot\be_h^{n-1}\|_0^2.
\end{eqnarray}
\end{lema}
\begin{proof} By adding $\pm b(\hat\bu_h^n,\bs_h^{n+1},\be_h^{n+1})$  and taking into account that
$b(\hat\bu_h^n,\be_h^{n+1},\be_h^{n+1})=0$, we have
$$
|b(\hat \bs_h^n ,\bs_h^{n+1},\be_h^{n+1})-b(\hat\bu_h^n,\bu_h^{n+1},\be_h^{n+1})|
=
|b(\hat\be_h^n,\bs_h^{n+1},\be_h^{n+1})|,
$$
which can be bounded as the term~$|b(\hat\be_h^n,\hat \bs_h^{n},\be_h^{n+1})|$ in the previous lemma.

\end{proof}

\begin{remark}\label{semimp-nl}\rm We notice the similarity between the bounds~in Lemmas~\ref{le:nl} and~\ref{le:nl-si}. Indeed, bounds~\eqref{eq:nonli1-0-si} and~\eqref{eq:nonli1-0_var-si} differ from~\eqref{eq:nonli1-0} and~\eqref{eq:nonli1-0_var} in that $\hat \bs_h^n$ in~\eqref{eq:nonli1-0} and~\eqref{eq:nonli1-0_var} is replaced
by~$\bs_h^{n+1}$ in \eqref{eq:nonli1-0-si} and~\eqref{eq:nonli1-0_var-si} and, more importantly, the last term,
$2\|\hat \bu_h^n\|_\infty c_{\rm inv} h^{-1}\|\be_h^{n+1}-\hat \be_h^n\|_0\|\be_h^{n+1}\|_0$ on the right-hand side
of~\eqref{eq:nonli1-0} and~\eqref{eq:nonli1-0_var} is not present in~\eqref{eq:nonli1-0-si} and~\eqref{eq:nonli1-0_var-si}.
We will see that this term gives rise to a CFL type condition that affects the IMEX~\eqref{fully_discrete} method but not the semi-implicit method~\eqref{semimp}.
\end{remark}

%

We now estimate the truncation errors.  For $\tau_2^{n+1}$, in view of~\eqref{eq:pressurel2} and~\eqref{eq:cotanewpro}, it easily follows that
 \begin{equation}
 \label{eq:norm_t2}
 \left\| \tau_2^{n+1}\right\|_0 \le C\bigl(\mu\left\| \bu^{n+1}\right\|_{k+1} h^{k} + \left\| p^{n+1}\right\|_{l+1}h^{l+1}\bigr).
 \end{equation}
For $\tau_1^{n+1}$, $\tau_3^{n+1}$ and~$\tau_4^{n+1}$ we treat separately the cases $n\ge 1$ and~$n=0$.

\begin{lema}\label{le:trun} The following bounds hold for $n\ge 1$.
\begin{align}
 \label{eq:norm_t1}
 \left\| \tau_1^{n+1}\right\|_0^2 &\le \frac{4h^{2k}}{\Delta t_n}\int_{t_n}^{t_{n+1}} \left\| \bu_t(t)\right\|_{k}^2 \,dt
 +\frac{2h^{2k}}{\Delta t_{n-1}} \int_{t_{n-1}}^{t_{n}} \left\| \bu_t(t)\right\|_{k}^2 \,dt
 \\
 &{}+ \frac{(1+\omega_n)^2}{10} (\Delta t_n)^3 \int_{t_{n}}^{t_{n+1}} \left\|\bu_{ttt}(t)\right\|_0^2\,ds
 +\frac{\omega_n^2}{10}(\Delta t_n + \Delta t_{n-1})^3\int_{t_{n-1}}^{t_{n+1}} \left\|\bu_{ttt}(t)\right\|_0^2\,dt.
 \nonumber
 \end{align}
 \begin{align}
\label{eq:cota_t3}
\left| (\tau_3^{n+1},\bv_h)\right| \le &C\Bigl((1+\omega_n)\left|\Omega\right|^{1/d}\left\|\bu\right\|_{L^\infty(W^{1,\infty})}h^k
\left\|\bu\right\|_{L^\infty(H^{k+1})}
\nonumber \\
&{}+ \left\|\bu\right\|_{L^\infty(W^{1,\infty})}
\left\|\hat\bu^n - \bu^{n+1}\right\|_0+ \left\|\bu\right\|_{L^\infty(L^\infty)}
\left\|\hat\bu^n - \bu^{n+1}\right\|_1\Bigr)\left\|\bv_h\right\|_0,
\end{align}
 \begin{align}
\label{eq:cota_t4}
\left| (\tau_4^{n+1},\bv_h)\right| \le &C\Bigl((1+\omega_n)\left|\Omega\right|^{1/d}\left\|\bu\right\|_{L^\infty(W^{1,\infty})}h^k
\left\|\bu\right\|_{L^\infty(H^{k+1})}
\nonumber \\
&{}+ \left\|\bu\right\|_{L^\infty(W^{1,\infty})}
\left\|\hat\bu^n - \bu^{n+1}\right\|_0\Bigr)\left\|\bv_h\right\|_0,
\end{align}

where
 \begin{align}
 \label{eq:extrapol_trl}
 \left\|  \bu^{n+1} -\hat \bu^{n}\right\|_l\le &\frac{1+\omega_n}{\sqrt{3}} (\Delta t_n)^{3/2}
 \biggl(\int_{t_{n}}^{t_{n+1}} \left\| \bu_{tt}(t)\right\|_l^2\,dt\biggr)^{1/2}
 \nonumber\\
 &{} +
\frac{\omega_n}{\sqrt{3}} (\Delta t_n+\Delta t_{n-1})^{3/2}
 \biggl(\int_{t_{n-1}}^{t_{n+1}} \left\| \bu_{tt}(t)\right\|_l^2\,dt\biggr)^{1/2},\qquad l=0,1.
 \end{align}
\end{lema}
\begin{proof}
We first write $\tau_1=(\Delta t_{n})^{-1}{\cal D}(\bs_h^{n+1}-\bu^{n+1}) + (\Delta t_{n})^{-1}({\cal D}\bu^{n+1} -
\bu_t^{n+1})$ and, further, we notice that we can write for the first term
\begin{align*}
\frac{1}{\Delta t_n} {\cal D}(\bs_h^{n+1}-\bu^{n+1})=& \Bigl(1+\frac{\omega_n}{1+\omega_n}\Bigr)\frac{1}{\Delta t_n} D (\bs_h^{n+1}-\bu^{n+1}) - \frac{\omega_n^2}{1+\omega_n} \frac{1}{\Delta t_{n}} D (\bs_h^{n}-\bu^{n})
\nonumber\\
{}= &\Bigl(1+\frac{\omega_n}{1+\omega_n}\Bigr)\frac{1}{\Delta t_n}\int_{t_n}^{t_{n+1}}\frac{d}{dt} (\bs_h(t)-\bu(t))\,dt
\nonumber\\
&{}
- \frac{\omega_n}{1+\omega_n} \frac{1}{\Delta t_{n-1}}\int_{t_{n-1}}^{t_{n}}\frac{d}{dt} (\bs_h(t)-\bu(t))\,dt.
\end{align*}
Furthermore, Taylor expansion with integral remainder shows that the second term can be bounded by
\begin{align*}
\frac{1}{\Delta t_n} ({\cal D}\bu^{n+1} -
\bu_t^{n+1})
 =& (1+\omega_n)\frac{1}{2\Delta t_n} \int_{t_{n}}^{t_{n+1}} (t-t_{n})^2\bu_{ttt}(t)\,dt\nonumber\\
 &{}-
 \frac{\omega_n^2}{2(1+\omega_n)}\frac{1}{\Delta t_n} \int_{t_{n-1}}^{t_{n+1}} (t-t_{n-1})^2\bu_{ttt}(t)\,dt.
\end{align*}
We also notice that
$$
 \frac{\omega_n^2}{2(1+\omega_n)}\frac{1}{\Delta t_n} =  \frac{\omega_n}{2(1+\omega_n)}\frac{1}{\Delta t_{n-1}}
 = \frac{\omega_n}{2(\Delta t_n + \Delta t_{n-1})}.
 $$
Thus, applying~\eqref{eq:cotanewpro} and H\"older's inequality an easy calculation shows that \eqref{eq:norm_t1} holds.

For $\tau_3^{n+1}$, adding $\pm b(\hat \bu^{n},\hat  \bs^n_h,\bv_h) \pm b(\hat\bu^n,\hat \bu^n,\bv_h) \pm
 b(\hat\bu^n, \bu^{n+1}, \bv_h)$, we can write
\begin{align}
\label{eq:cota_t3_0}
 (\tau_3^{n+1},\bv_h) = &b(\hat\bs_h^n-\hat\bu^n,\hat\bs_h^n,\bv_h)+
 b(\hat \bu^n,\hat\bs_h^n-\hat\bu^n,\bv_h) + b(\hat \bu^n,\hat\bu^n-\bu^{n+1},\bv_h)
 \nonumber\\
 &{}+ b(\hat\bu^n-\bu^{n+1},\bu^{n+1},\bv_h).
\end{align}
Thus, by noticing that
$\left\|\hat\bs_h^n - \hat \bu^{n}\right\|_l\le 2(1+\omega_n)\left\|\bs_h-\bu\right\|_{L^\infty(H^l)}$, $l=0,1$, applying~\eqref{eq:cotanewpro}, H\"older's inequality and bounding $h\le \left|\Omega\right|^{1/d}$, when necessary, we get \eqref{eq:cota_t3}.
For $\tau_4$, adding $\pm b(\hat \bu^{n}, \bs_h^{n+1},\bv_h) \pm b(\hat\bu^n,\bs_h^{n+1},\bv_h)$ one gets
\begin{align}
\label{eq:cota_t4_0}
 (\tau_4^{n+1},\bv_h) = &b(\hat\bs_h^n-\hat\bu^n,\bs_h^{n+1},\bv_h)+
 b(\hat \bu^n-\bu^{n+1},\bs_h^{n+1},\bv_h) + b(\bu^{n+1},\bs_h^{n+1}-\bu^{n+1},\bv_h),
\end{align}
from where, arguing as with~$\tau_3^{n+1}$, the bound~\eqref{eq:cota_t4} follows.

 Also, by means of Taylor expansion with integral remainder and applying H\"older's inequality, it is easy to show \eqref{eq:extrapol_trl}.
  \end{proof}

We now deal with the truncation errors in the first step. Recall  that
$$
\tau_1^{1}=\frac{1}{\Delta t_0}{ D}\bs_h^{1}-\bu_t^{1}
$$
and $\tau_3^1$ are defined as in \eqref{lostrun} but with $\hat \bs_h^0=\bs_h^0$.
\begin{lema}\label{le:trun1} We have the following bounds,
\begin{align}
\label{trunc11}
\|\tau_1^1\|_0^2\le& C\Bigl(\frac{1}{\Delta t_0}h^{2k}\|\bu_t\|_{L^2(H^k)} +(\Delta t_0)^2\|\bu_{tt}\|_{L^\infty(L^2)}^2\Bigr),
\\
\label{trunc13}
\left\|\tau_3^1\right\|_0^2\le &C\Bigl( \left|\Omega\right|^{2/d}\left\|\bu\right\|_{L^\infty(W^{1,\infty})}^2h^{2k}\nonumber\\
&\quad{}+(\Delta t_0)^2 \bigl( \left\|\bu\right\|_{L^{\infty}(W^{1,\infty})}^2
\left\| \bu_t\right\|_{L^\infty(L^2)}^2 + \left\|\bu\right\|_{L^{\infty}(L^\infty)}^2
\left\| \bu_t\right\|_{L^\infty(H^1)}^2\bigr)\Bigr).
\end{align}
\end{lema}
\begin{proof}
For~$\tau_1^1$, by arguing as in \eqref{eq:norm_t1} we have
$$\|\tau_1^1\|_0^2\le \frac{C}{\Delta t_0}h^{2k}\int_{t_0}^{t_1}\|\bu_t\|_k^2 \,dt+C \Delta t_0\int_{t_0}^{t_1}\|\bu_{tt}\|_0^2 ,\,dt,
$$
from where~\eqref{trunc11} follows.
To prove~\eqref{trunc13} we go back to~\eqref{eq:cota_t3_0}, and repeat the argument there, but now taking into account that
for~$n=0$, $\hat\bs_h^0=\bs_0$ and~$\hat\bu^0=\bu^0$, so that $\left\|\hat \bs_h^0 -\bu^0\right\|_l\le Ch^{k+1-l}\left\|\bu^0\right\|_{k+1}$ and
$$
\left\| \hat \bu^0 -\bu^1\right\|_l =\biggl\| \int_{t_0}^{t_1} \bu_t\,dt\biggr\|_l \le \Delta t_0 \left\|\bu_t\right\|_{L^\infty(H^l)}.
$$
Then~\eqref{trunc13} follows easily. 
\end{proof}


\subsection{Fixed stepsize}\label{se:fixed}
We consider now the case where $\Delta t_n=\Delta t=T/N$, $n=0,1,\ldots,N-1$.
\begin{Theorem}\label{th1} Fix $\kappa>0$ and let $\Delta t$ and $h$ satisfy the following CFL-type condition
\begin{eqnarray}\label{CFL-u}
64c_{\rm inv}^2\left\|\bu\right\|_{L^\infty(L^\infty)}^2\frac{\Delta t}{h^2}\le \frac{\kappa}{T},
\end{eqnarray}
and
\begin{eqnarray}\label{delta_t_fixed}
\Delta t\Bigl(4\frac{\kappa+1}{T}+ M^k\Bigr)\le \frac{1}{2},\quad k=2,\ldots,n+1,
\end{eqnarray}
with $M^k$ defined in \eqref{Mjota}.
Then, there exist a positive constant $h_0$, such that for $h<h_0$ the error $\be_h^{n+1}=\bs_h^{n+1}-\bu_h^{n+1}$ of the IMEX method~\eqref{fully_discrete}, satisfies the following bound for $1\le n\le N-1=T/\Delta t -1$,
\begin{eqnarray}\label{er:cota_fin_th}
&&\|\be_h^{n+1}\|_0^2+\sum_{k=2}^{n+1}\nu \Delta t \|\nabla \be_h^k\|_0^2
+\mu\sum_{k=2}^{n+1}\Delta t \|\nabla\cdot \be_h^k\|_0^2
\le \nonumber\\
&&\quad  e^{8(\kappa+1) +2T\max_jM^j}
\left(\left(C_0^2+TC_3^2\right)(\Delta t)^4 + T\left(C_1^2h^{2k} + C_2^2h^{2l+2}\right)\right),
\end{eqnarray}
where
the constants $C_i^2$ are defined in
\eqref{eq:C2} and \eqref{eq:C0}.
\end{Theorem}
\begin{proof}
Taking $\bv_h=\be_h^{n+1}$ in \eqref{fully_error} we obtain
\begin{eqnarray}\label{eq:er2}
&&\frac{1}{4}\frac{1}{\Delta t}\|\be_h^{n+1}\|_0^2+\frac{1}{4}\frac{1}{\Delta t}\|\hat \be_h^{n+1}\|_0^2
+\frac{1}{4}\frac{1}{\Delta t }\|\be_h^{n+1}-\hat \be_h^n\|_0^2
\\
&&\quad-\frac{1}{4}\frac{1}{\Delta t }\|\be_h^n\|_0^2
-\frac{1}{4}\frac{1}{\Delta t }\|\hat \be_h^n\|_0^2
+\nu\|\nabla \be_h^{n+1}\|_0^2
\nonumber\\
&&\quad +\mu\|\nabla \cdot \be_h^{n+1}\|_0^2\le |b(\hat \bs_h^n,\hat \bs_h^n,\be_h^{n+1})-b(\hat\bu_h^n,\hat\bu_h^n,\be_h^{n+1})|
+\frac{T}{2}\|\tau_1^{n+1}\|_0^2+\frac{T}{2}\|\tau_3^{n+1}\|_0^2\nonumber\\
&&\quad+\frac{1}{T}\|\be_h^{n+1}\|_0^2+\frac{9}{2\mu}\|\tau_2^{n+1}\|_0^2+\frac{\mu}{18}\|\nabla \cdot\be_h^{n+1}\|_0^2.
\end{eqnarray}
We now apply \eqref{eq:nonli1-0} from Lemma~\ref{le:nl} and, further, we bound
\begin{equation}
\label{eq:cota_ref}
2\|\hat \bu_h^n\|_\infty c_{\rm inv} h^{-1}\|\be_h^{n+1}-\hat e_h^n\|_0\|\be_h^{n+1}\|_0
\le 
4c_{\rm inv}^2\|\hat \bu_h^n\|_\infty^2\frac{\Delta t}{h^2}\|\be_h^{n+1}\|_0^2 +
\frac{\|\be_h^{n+1}-\hat \be_h^n\|_0^2}{4\Delta t}.
\end{equation}
Then
\begin{eqnarray}\label{eq:nonli1}
|b(\hat \bs_h^n,\hat \bs_h^n,\be_h)-b(\hat\bu_h^n,\hat\bu_h^n,\be_h)|&\le&
\left(\frac{\|\nabla \hat \bs_h^n\|_\infty}{2}+\frac{9}{2\mu}\|\hat \bs_h^n\|_\infty^2+4c_{\rm inv}^2\|\hat \bu_h^{n}\|_\infty^2\frac{\Delta t}{h^2}\right)\|\be_h^{n+1}\|_0^2\nonumber\\
&&\quad +3\|\nabla \hat \bs_h^n\|_\infty\|\be_h^n\|_0^2+\frac{3}{2}\|\nabla \hat \bs_h^n\|_\infty\|\be_h^{n-1}\|_0^2
\\
&&\quad+\frac{\mu}{3}\|\nabla \cdot\be_h^n\|_0^2+\frac{\mu}{6}\|\nabla \cdot\be_h^{n-1}\|_0^2+\frac{\|\be_h^{n+1}-\hat \be_h^n\|_0^2}{4\Delta t}.
\nonumber
\end{eqnarray}
The last term on the right-hand side of \eqref{eq:nonli1} will be absorbed with the left-hand side of \eqref{eq:er2}. In the sequel, we assume that, for $0\le n\le N-1$,
\begin{eqnarray}\label{CFL}
4c_{\rm inv}^2\|\hat \bu_h^{n}\|_\infty^2\frac{\Delta t}{h^2}\le \frac{\kappa}{T}.
\end{eqnarray}
At the end of the proof we will show that~(\ref{CFL}) always holds for $h$ sufficiently small.

Inserting \eqref{eq:nonli1}
into 	\eqref{eq:er2} we reach
\begin{eqnarray}\label{eq:er3}
&&\frac{1}{4}\frac{1}{\Delta t}\|\be_h^{n+1}\|_0^2+\frac{1}{4}\frac{1}{\Delta t}\|\hat \be_h^{n+1}\|_0^2
-\frac{1}{4}\frac{1}{\Delta t }\|\be_h^n\|_0^2
-\frac{1}{4}\frac{1}{\Delta t }\|\hat \be_h^n\|_0^2
\nonumber\\
&&\quad+\nu\|\nabla \be_h^{n+1}\|_0^2 +\frac{17}{18}\mu\|\nabla \cdot \be_h^{n+1}\|_0^2\le 
\\
&&\left(\frac{\|\nabla \hat \bs_h^n\|_\infty}{2}+\frac{9}{2\mu}\|\hat \bs_h^n\|_\infty^2+\frac{\kappa+1}{T}\right)\|\be_h^{n+1}\|_0^2
+3\|\nabla \hat \bs_h^n\|_\infty\|\be_h^n\|_0^2+\frac{3}{2}\|\nabla \hat \bs_h^n\|_\infty\|\be_h^{n-1}\|_0^2
\nonumber\\
&&\quad+\frac{\mu}{3}\|\nabla \cdot\be_h^n\|_0^2+\frac{\mu}{6}\|\nabla \cdot\be_h^{n-1}\|_0^2+\frac{T}{2}\|\tau_1^{n+1}\|_0^2+\frac{T}{2}\|\tau_3^{n+1}\|_0^2+\frac{9}{2\mu}\|\tau_2^{n+1}\|_0^2.\nonumber
\end{eqnarray}
Taking the sum of terms in \eqref{eq:er3} we get
\begin{eqnarray}\label{eq:er4}
&&\frac{1}{4}\|\be_h^{n+1}\|_0^2+\frac{1}{4}\|\hat \be_h^{n+1}\|_0^2+\sum_{k=2}^{n+1}\nu\Delta t \|\nabla \be_h^k\|_0^2
+\frac{4\mu}{9}\sum_{k=2}^{n+1}\Delta t \|\nabla \cdot \be_h^k\|_0^2
\le \frac{1}{4}\|\be_h^1\|_0^2+\frac{1}{4}\|\hat \be_h^1\|_0^2\nonumber\\
&&\quad +\frac{\mu}{2}(\Delta t)\|\nabla \cdot \be_h^1\|_0^2
+\frac{\mu}{6}(\Delta t)\|\nabla \cdot \be_h^0\|_0^2
+\frac{3}{2}(\Delta t)\|\nabla \hat \bs_h^1\|_\infty\|\be_h^0\|_0^2\\
&&\quad+(\Delta t)\left(\frac{3}{2}\|\nabla \hat \bs_h^2\|_\infty
+3\|\nabla \hat \bs_h^1\|_\infty\right)\|\be_h^1\|_0^2
+\sum_{k=2}^{n+1}\Delta t\left(\frac{T}{2}\|\tau_1^k\|_0^2+\frac{T}{2}\|\tau_3^k\|_0^2+\frac{9}{2\mu}\|\tau_2^k\|_0^2\right)
\nonumber\\
&&\quad+\sum_{k=2}^{n+1}\Delta t \left(\frac{\|\nabla \hat \bs_h^{k-1}\|_\infty}{2}+\frac{9}{2\mu}\|\hat \bs_h^{k-1}\|_\infty^2
+\frac{ \kappa+1}{T}
+3\|\nabla \hat \bs_h^k\|_\infty+\frac{3}{2}\|\nabla \hat \bs_h^{k+1}\|_\infty\right)\|\be_h^k\|_0^2,\nonumber
\end{eqnarray}
where in the last term of \eqref{eq:er3} to simplify writing we assume $\|\nabla \hat \bs_h^{k+1}\|_\infty=0$ for $k=n+1$.
In the sequel, to simplify notation we define
\begin{equation}\label{Mjota}
M^k=4\biggl(\frac{\|\nabla \hat \bs_h^{k-1}\|_\infty}{2}+\frac{9}{2\mu}\|\hat \bs_h^{k-1}\|_\infty^2
+3\|\nabla \hat \bs_h^k\|_\infty+\frac{3}{2}\|\nabla \hat \bs_h^{k+1}\|_\infty\biggr)
\end{equation}
and
\begin{eqnarray}\label{E0}
E^0&=&4\biggl(\frac{1}{4}\|\be_h^1\|_0^2+\frac{1}{4}\|\hat \be_h^1\|_0^2+\frac{\mu}{2}(\Delta t)\|\nabla \cdot \be_h^1\|_0^2
+\frac{\mu}{6}(\Delta t)\|\nabla \cdot \be_h^0\|_0^2
+\frac{3}{2}(\Delta t)\|\nabla \hat \bs_h^1\|_\infty\|\be_h^0\|_0^2\nonumber\\
&&\quad+(\Delta t)\left(\frac{3}{2}\|\nabla \hat \bs_h^2\|_\infty
+3\|\nabla \hat \bs_h^1\|_\infty\right)\|\be_h^1\|_0^2\biggr).
\end{eqnarray}
Let us observe that all the terms in \eqref{Mjota} can be bounded in terms of $\|\bu\|_{L^\infty(L^\infty)}$ and
$\|\nabla \bu\|_{L^\infty(L^\infty)}$ applying \eqref{cotainfty0} and \eqref{cotainfty1}.
 
With the above simplifications we can write equation \eqref{eq:er4} in the form
\begin{align}\label{eq:er5}
\|\be_h^{n+1}\|_0^2+\sum_{k=2}^{n+1}\nu \Delta t \|\nabla \be_h^k\|_0^2
+\mu\sum_{k=2}^{n+1}\Delta t \|\nabla\cdot \be_h^k\|_0^2&
\le E^0+\sum_{k=2}^{n+1}\Delta t\Bigl(4\frac{\kappa+1}{T}+  M^k\Bigr)\|\be_h^k\|_0^2\\
&{}+2\sum_{k=2}^{n+1}\Delta t \left(T(\|\tau_1^k\|_0^2+\|\tau_3^k\|_0^2)+\frac{9}{\mu}\|\tau_2^k\|_0^2\right).\nonumber
\end{align}
Assuming \eqref{delta_t_fixed} we can apply Gronwall Lemma \ref{gronwall} to obtain
\begin{eqnarray}\label{er:cota_fin}
&&\|\be_h^{n+1}\|_0^2+\sum_{k=2}^{n+1}\nu \Delta t \|\nabla \be_h^k\|_0^2
+\mu\sum_{k=2}^{n+1}\Delta t \|\nabla\cdot \be_h^k\|_0^2
\le \nonumber\\
&&\quad  e^{8(\kappa+1) +2T\max_jM^j}\left(E^0+2\sum_{k=2}^{n+1}\Delta t \left(T(\|\tau_1^k\|_0^2+\|\tau_3^k\|_0^2)+\frac{9}{\mu}\|\tau_2^k\|_0^2\right)\right).
\end{eqnarray}
In view of~\eqref{eq:norm_t2} and \eqref{eq:norm_t1}-\eqref{eq:cota_t3}-\eqref{eq:extrapol_trl} 
from Lemma \ref{le:trun} and taking into account that $\omega_n=1$, $n=1,\ldots,N=T/\Delta t$, we have
$$
\Delta t \sum_{k=2}^n\left\|\tau_1^k\right\|_0^2 \le C\Bigl(\left\| \bu_t\right\|_{L^2(H^k)}^2h^{2k} + 
\left\|\bu_{ttt}\right\|_{L^2(L^2)}^2(\Delta t)^4\Bigr),
$$
$$
\Delta t \sum_{k=2}^n\left\|\tau_2^k\right\|_0^2 \le CT\left(\mu^2 \left\|\bu\right\|_{L^\infty(H^{k+1})}^2 h^{2k} + 
\left\| p\right\|_{L^\infty(H^{l+1})}^2h^{2(l+1)}\right),
$$
and
\begin{align*}
\Delta t \sum_{k=2}^n\left\|\tau_3^k\right\|_0^2\le &C\left\| \bu\right\|_{L^\infty(W^{1,\infty})}^2\left(T
\left|\Omega\right|^{2/d}\left\|\bu\right\|_{L^\infty(H^{k+1})}^2h^{2k}
+ \left\|\bu_{tt}\right\|_{L^2(L^2)}^2(\Delta t)^4
\right)
\nonumber\\
&+C\left\| \bu\right\|_{L^\infty(L^\infty)}^2
 \left\|\bu_{tt}\right\|_{L^2(H^1)}^2(\Delta t)^4.
\end{align*}
Consecuently,
\begin{equation}
\label{eq:suma_truncs}
\sum_{k=2}^{n+1}\Delta t \left(T(\|\tau_1^k\|_0^2+\|\tau_3^k\|_0^2)+\frac{9}{\mu}\|\tau_2^k\|_0^2\right)
\le T\left(C_1^2h^{2k} + C_2^2h^{2l+2} + C_3^2(\Delta t)^4\right),
\end{equation}
where
\begin{align}
C_1^2&=C\left(\left\|\bu_t\right\|_{L^2(H^k)}^2+ \left(\mu +T\left|\Omega\right|^{2/d}\left\|\bu\right\|_{L^\infty(W^{1,\infty})}^2\right) \left\|\bu\right\|_{L^\infty(H^{k+1})}^2\right),
\nonumber\\
\label{eq:C2}
C_2^2&=C\frac{1}{\mu}\left\|p\right\|_{L^\infty(H^{l+1})}^2,
\\
\nonumber
C_3^2&=C\left(\left\|\bu_{ttt}\right\|_{L^2(L^2)}^2 + \left\|\bu\right\|_{L^\infty(W^{1,\infty})}^2\left\|\bu_{tt}\right\|_{L^2(L^2)}^2 + \left\|\bu\right\|_{L^\infty(L^\infty)}^2\left\|\bu_{tt}\right\|_{L^2(H^1)}^2
\right).
\end{align}
Finally, since we have assumed $\bu_h^0=\bs_h^0$ to bound $E_0$ we only need to bound $\be_h^1$. To this end, let us observe that after one step
using Euler method with the explicit form of the nonlinear term, and arguing as before, it is easy to get
\begin{equation}\label{eq:ert1}
\frac{1}{\Delta t_0}\|\be_h^1\|_0^2+\nu\|\nabla \be_h^1\|_0^2+\mu\|\nabla \cdot \be_h^1\|_0^2=(\tau_1^1+\tau_3^1,\be_h^1)
+(\tau_2^1,\nabla \cdot \be_h^1).
\end{equation}
From \eqref{eq:ert1} we get
$$
\frac{1}{2}\|\be_h^1\|_0^2+\frac{\mu}{2}\Delta t_0\|\nabla \cdot \be_h^1\|_0^2\le \frac{(\Delta t_0)^2}{2}\|\tau_1^1+\tau_3^1\|_0^2
+\frac{\Delta t_0}{2\mu}\|\tau_2^1\|_0^2.
$$
Now, in view of~(\ref{trunc11}) and  taking into account that $\Delta t\le T$, we have
\begin{equation}
\label{eq:Dt12}
(\Delta t)^2\|\tau_1^1\|_0^2
 \le C \|\bu_{tt}\|_{L^\infty(L^2)}^2(\Delta t)^4 + TC_1^2h^{2k},
\end{equation}
where~$C_1$ is the constant in~\eqref{eq:C2}.
Also in view of~\eqref{trunc13} we have
$$
(\Delta t)^2 \left\|\tau_3^1\right\|_0^2\le C(\Delta t)^4 \bigl( \left\|\bu\right\|_{L^{\infty}(W^{1,\infty})}^2
\left\| \bu_t\right\|_{L^\infty(L^2)}^2 + \left\|\bu\right\|_{L^{\infty}(L^\infty)}^2
\left\| \bu_t\right\|_{L^\infty(H^1)}^2\bigr) + TC_1 h^{2k}.
$$
Thus, we finally obtain
\begin{equation}\label{cota_ini}
\|\be_h^1\|_0^2+\mu\Delta t\|\nabla \cdot \be_h^1\|_0^2\le 
C_0^2(\Delta t)^4 + T\left(C_1^2h^{2k} + C_2^2h^{2l+2} \right),
\end{equation}
where
\begin{equation}
\label{eq:C0}
C_{0}^2=C\bigl( \|\bu_{tt}\|_{L^\infty(L^2)}^2
+\left\|\bu\right\|_{L^{\infty}(W^{1,\infty})}^2
\left\| \bu_t\right\|_{L^\infty(L^2)}^2 + \left\|\bu\right\|_{L^{\infty}(L^\infty)}^2
\left\| \bu_t\right\|_{L^\infty(H^1)}^2\bigr).
\end{equation}
Going back to the definition of $E_0$, \eqref{E0}, we reach
$$
E_0\le \left(5+\Delta t\left[6\|\nabla \hat \bs_h^2\|_\infty
+12\|\nabla \hat \bs_h^1\|_\infty\right]\right)\|\be_h^1\|_0^2+2\mu(\Delta t)\|\nabla\cdot \be_h^1\|_0^2,
$$
so that applying \eqref{cota_ini} and~\eqref{cotainfty1} we finally obtain
\begin{equation}\label{cota_E0}
E_0 \le  C\left(1+\Delta t\left\|\bu\right\|_{L^\infty(W^{1,\infty})} \right)
\left(C_0^2(\Delta t)^4 + T\left(C_1^2h^{2k} + C_2^2h^{2l+2}\right)\right).
\end{equation}
Inserting \eqref{eq:suma_truncs}, \eqref{cota_E0} into \eqref{er:cota_fin} we conclude \eqref{er:cota_fin_th} as long as condition~\eqref{CFL} holds. Thus, in order to finish the proof, we have to check that this is the case if $h$ and~$\Delta t$ are taken sufficiently small. To this end we devote the rest of the proof. 

We will show that for $h$ sufficiently small it holds
\begin{equation}
\label{conti0}
\|\hat\bu_h^n\|_\infty \le 4\left\|\bu\right\|_{L^\infty(L^\infty)},
\end{equation}
for $0\le n\le N-1$, so that~\eqref{CFL} will be a consequence of~\eqref{CFL-u}.
For this purpose, we start by taking $h_{0,1}$ such that the right-hand side of~\eqref{cotainfty0} is bounded by~$\left\|\bu\right\|_{L^\infty(L^\infty)}/6$,
so that by writing~$\bs_h=\bs_h - \bu + \bu$, and
in view of~\eqref{cotainfty0}, we have that, if $h\le h_{0,1}$, then,
$\left\| \bs_h\right\|_{L^\infty(L^\infty)}\le (7/6)\left\|\bu\right\|_{L^\infty(L^\infty)}$, and, consequently,
\begin{equation}
\label{conti1}
\left\|\hat \bs_h^n \right\| \le (7/2) \left\|\bu\right\|_{L^\infty(L^\infty)}, \qquad 0\le t_n\le T.
\end{equation}
For~$\|\hat\bu_h^n\|_\infty$, 
adding and subtracting
$\hat\bs_h^n$ and using inverse inequality \eqref{inv}  we get
\begin{align}
\label{conti2}
\|\hat\bu_h^n\|_\infty&\le \|\hat\bu_h^n-\hat \bs_h^n\|_\infty+\|\hat\bs_h^n\|_\infty
\le c_{\rm inv}h^{-d/2}\|\hat\bu_h^n-\hat \bs_h^n\|_0+\|\hat\bs_h^n\|_\infty\nonumber\\
&\le
c_{\rm inv}h^{-d/2}2\|\bu_h^n- \bs_h^n\|_0+c_{\rm inv}h^{-d/2}\|\bu_h^{n-1}- \bs_h^{n-1}\|_0+(7/2) \left\|\bu\right\|_{L^\infty(L^\infty)},
\end{align}
where, in the last inequality, we have also applied~\eqref{conti1}.
We now take $h_{0,2}\le h_{0,1}$ such that if $h\le h_{0,2}$ the right-hand sides of~\eqref{er:cota_fin_th}
and~\eqref{cota_ini} are smaller that
\begin{equation}
\label{cantidad}
r(h)=\frac{h^{d}}{36c_{\rm inv}^2}.
\end{equation}
Notice that this is possible since, on the one hand, 
due to~\eqref{CFL-u}, $\Delta t$ can be bounded in terms of~$ h^2$, and, on the other hand, the right-hand sides of~~\eqref{er:cota_fin_th}
and~\eqref{cota_ini}, being at least $O(h^4)$, decay faster with~$h$ than~$r(h)$ in~\eqref{cantidad}, which is, at most, $O(h^3)$.

For $h\le h_{0,2}$, and since we are taking $\bu_h^0=\bs_h^0$, in view of~(\ref{conti2}), we have that
\eqref{conti0}, holds for~$n=0,1$. Assuming that \eqref{conti0} holds for $n\le m$, we will now show that it also holds
for~$n=m+1$, and this will finish the proof. Indeed, if~(\ref{conti0}) holds for $n\le m$, then, as argued above,
\eqref{er:cota_fin_th} holds for~$n=m$. But, since~$h\le h_{0,2}$, the right-hand side of~\eqref{er:cota_fin_th} is smaller that~$r(h)$ in~\eqref{cantidad}, and, consequently, for $n=m+1$, the right-hand side of~\eqref{conti2} is smaller that
$(7/2+1/2) \left\|\bu\right\|_{L^\infty(L^\infty)}$, that is, \eqref{conti0} also holds for~$n=m+1$.
\end{proof}
\begin{remark}\label{re:2}
In view of~\eqref{CFL-u}, we observe that as $\Delta t,h\rightarrow 0$ the ratio $\Delta t/h^2$ must be bounded. One can allow the upper bound to be larger or smaller, but its size reflects in the error bound~\eqref{er:cota_fin_th} through the factor
$\exp(8(\kappa+1))$ on the right-hand side of the error bound \eqref{er:cota_fin_th}.

As pointed out in the introduction, arguing as in \cite{imex_schneier}, it is possible to get a different CFL condition in which the quantity that has to be bounded is $(\Delta t) h^{-1}\nu^{-1}$ instead of $(\Delta t)h^{-2}$ and that involves a weaker norm for $\bu$ in \eqref{CFL-u}. However, since we focus on getting bounds that hold for high Reynolds numbers we will develop this line of research in future works.
\end{remark}
\bigskip

For the semi-implicit method~\eqref{semimp} we have the following result.

\begin{Theorem}\label{th1-si} Let $\Delta t$ satisfy the following bound for $k=2,\ldots,n+1$,
$$
\Delta t\Bigl(\frac{1}{T}+ N^k\Bigr)\le \frac{1}{2},\quad k=2,\ldots,n+1,
$$
where
$$
N_k
=4\biggl(\frac{\|\nabla \bs_h^{k}\|_\infty}{2}+\frac{9}{2\mu}\| \bs_h^{k}\|_\infty^2
+3\|\nabla  \bs_h^{k+1}\|_\infty+\frac{3}{2}\|\nabla \bs_h^{k+2}\|_\infty\biggr).
$$
Then, there exist a positive constant $h_0$, such that for $h<h_0$ the error $\be_h^{n+1}=\bs_h^{n+1}-\bu_h^{n+1}$ of the semi-implicit method~\eqref{semimp}, satisfies the following bound for $1\le n\le N-1=T/\Delta t -1$,
\begin{eqnarray*}
&&\|\be_h^{n+1}\|_0^2+\sum_{k=2}^{n+1}\nu \Delta t \|\nabla \be_h^k\|_0^2
+\mu\sum_{k=2}^{n+1}\Delta t \|\nabla\cdot \be_h^k\|_0^2
\le \nonumber\\
&&\quad  e^{8 +2T\max_jN^j}
\left(\left(C_0^2+TC_3^2\right)(\Delta t)^4 + T\left(C_1^2h^{2k} + C_2^2h^{2l+2}\right)\right),
\end{eqnarray*}
where
the constants $C_i^2$ are defined in
\eqref{eq:C2} and \eqref{eq:C0}.
\end{Theorem}

\begin{proof} Arguing as in the case of method~\eqref{fully_discrete}, we have that~\eqref{eq:er2} also holds for the semi-implicit method but with the first term on the right-hand side of~\eqref{eq:er2} replaced by
 $$
 |b(\hat \bs_h^n,\bs_h^{n+1},\be_h^{n+1})-b(\hat\bu_h^n,\bu_h^{n+1},\be_h^{n+1})|
 $$
 and~$\tau_3^n$ replaced by~$\tau_4^n$. Then, recalling~Lemma~\ref{le:nl-si}
 and taking into account Remark~\ref{semimp-nl}, we have that~\eqref{eq:er3} also holds true for method~\eqref{semimp} with $\kappa=0$, $\hat \bs_h^n$ replaced by~$\bs_h^{n+1}$ and $\tau_3^n$ replaced by~$\tau_4^n$. Then the proof follows the steps of the proof of Theorem~\ref{th1}. In the present case, there is no need to argue about the boundedness of~$\bu_h^n$ since it is not necessary that~condition~\eqref{CFL} is satisfied.
\end{proof}

\begin{remark}\label{re_order} From Theorems \ref{th1} and~\ref{th1-si}, using triangle inequality and applying \eqref{eq:cotanewpro}, we reach
$$
\max_{1\le n\le N}\|\bu_h^n-\bu^n\|_0=O((\Delta t)^2+h^k+h^{l+1}),
$$
which means that the rate of convergence is of optimal order 2 in time. Choosing for example Hood-Taylor elements with $l=k-1$ 
we get a rate of convergence of order $k$ in space, which matches previously results in \cite{nos_grad_div}, as stated in the introduction.
\end{remark}
\subsection{Variable stepsize}\label{se:var}
In this section we carry out the error analysis for the variable stepsize case. 

We start with a technical lemma that is needed to prove the main result.
\begin{lema}\label{le:tech}
Assuming condition \eqref{eq:cond_omega}
and denoting by
\begin{align}
\label{eq:elG_n}
G_n&=\frac{\omega_{n-1}}{2(1+\omega_{n-1})}\left\|\be_h^{n}\right\|_0^2 +  \frac{1}{4}\left\| \be_h^{n}+(\be_h^{n}-\be_h^{n-1})\right\|_0^2,\qquad n=1,2,\ldots,N,
\end{align}
the following inequality holds, for a constant $K_0\le 75/14$, 
\begin{align}
\label{eq:pvar1_d}
\Delta t_n({\cal D} \be_h^{n+1},\be_h^{n+1}) 
 &\ge G_{n+1} - G_n+\frac{\omega_n}{2(1+\omega_n)}\left\|  \be_h^{n+1}-\hat\be_h^n\right\|_0^2  \\
 &\quad - \frac{K_0}{\gamma}(1+\Gamma)^2 \left(\left|\omega_{n-1}-1\right| + \left|\omega_n-1\right|\right)G_n.\nonumber
\end{align}
\end{lema}
\begin{proof}
We first rewrite adequately the term~$({\cal D}\be_h^{n+1},\be_h^{n+1})$.
Using the identity
\begin{equation}
\label{eq:the_identity}
(\bv-\bw,\bv)=\frac{1}{2}\bigl(\left\| \bv\right\|_0^2 -\left\| \bw\right\|_0^2+\left\| \bv-\bw\right\|_0^2),
\end{equation}
we have that
\begin{align}
\label{eq:pvar0}
\Delta t_n({\cal D} \be_h^{n+1},\be_h^{n+1}) =& \frac{1}{2}
\bigl(\left\| \be_h^{n+1}\right\|_0^2 -\left\| \be_h^{n}\right\|_0^2+\left\| \be_h^{n+1}-\be_h^{n}\right\|_0^2\bigr)\\
\nonumber
&{}
+ \frac{\omega_n}{2(1+\omega_n)}
\bigl(\left\| \be_h^{n+1}\right\|_0^2 -\left\| \hat \be_h^{n}\right\|_0^2+\left\| \be_h^{n+1}-\hat\be_h^{n}\right\|_0^2\bigr).
\end{align}
For the first term on the right-hand side above, using the identity~\eqref{eq:the_identity}, we write
\begin{align*}
\frac{1}{2}
\bigl(\left\| \be_h^{n+1}\right\|_0^2 -\left\| \be_h^{n}\right\|_0^2+\left\| \be_h^{n+1}-\be_h^{n}\right\|_0^2\bigr)=&
\frac{1}{4}
\bigl(\left\| \be_h^{n+1}\right\|_0^2 -\left\| \be_h^{n}\right\|_0^2+\left\| \be_h^{n+1}-\be_h^{n}\right\|_0^2\bigr)
\\
&{}+\frac{1}{2}(\be_h^{n+1},\be_h^{n+1}-\be_h^{n})
\\
=&\frac{1}{4}\left\| \be_h^{n+1}+(\be_h^{n+1}-\be_h^{n})\right\|_0^2 - \frac{1}{4}\left\| \be_h^{n}\right\|_0^2.
\end{align*}
Then
\begin{align}
\label{eq:pvar1}
\Delta t_n({\cal D} \be_h^{n+1},\be_h^{n+1}) &=G_{n+1} - \Bigl(\frac{1}{4}\left\| \be_h^{n}\right\|_0^2 + 
 \frac{\omega_n}{2(1+\omega_n)}\left\| \hat\be_h^{n}\right\|_0^2\Bigr) +
\frac{\omega_n}{2(1+\omega_n)}\left\| \be_h^{n+1}-\hat\be_h^n\right\|_0^2 \nonumber\\
 &= G_{n+1} - G_n +{\cal R}_n +  \frac{\omega_n}{2(1+\omega_n)}\left\|  \be_h^{n+1}-\hat\be_h^n\right\|_0^2,
\end{align}
where
$$
{\cal R}_n  = G_{n}-\Bigl(\frac{1}{4}\left\| \be_h^{n}\right\|_0^2 + 
 \frac{\omega_n}{2(1+\omega_n)}\left\| \hat\be_h^n\right\|_0^2\Bigr).
$$
Noticing that
$$
\left\| \be_h^{n}+\omega_n(\be_h^{n}-\be_h^{n-1})\right\|_0^2= (1+\omega_n)^2\left\| \be_h^{n}\right\|_0^2
-2\omega_n(1+\omega_n) (\be_h^{n},\be_h^{n-1}) + \omega_{n}^2 \left\| \be_h^{n-1}\right\|_0^2,
$$
a straightforward computation shows
$$
 {\cal R}_n = c_{1,n} \left\| \be_h^{n}\right\|_0^2 - c_{2,n}(\be_h^{n},\be_h^{n-1}) +
c_{3,n}\left\| \be_h^{n-1}\right\|_0^2,
$$
where
\begin{align}
c_{1,n} & =\frac{\omega_{n-1}}{2(1+\omega_{n-1})}-\frac{1}{4}+ 1-\frac{\omega_n(1+\omega_n)}{2},
\label{eq:c_1n}\\
c_{2,n} &= 1 -\omega_n^2,
\label{eq:c_2n}\\
c_{3,n} & =\frac{1}{4}- \frac{\omega_n^3}{2(1+\omega_n)}.
\end{align}
We now express these three coefficients in terms of $\omega_n-1$ and~$\omega_{n-1}-1$. For $c_{1,n}$, we have
$$
c_{1,n}=\frac{\omega_{n-1}-1}{4(1+\omega_{n-1})}+\frac{2-\omega_n-\omega_n^2}{2}
=
\frac{\omega_{n-1}-1}{4(1+\omega_{n-1})} + \frac{2+\omega_n}{2}(1-\omega_n).
$$
For $c_{3,n}$, we have
$$
c_{3,n}=\frac{1+\omega_n-2\omega_n^3}{4(1+\omega_n)}
=\frac{1+2\omega_n+2\omega_n^2}{4(1+\omega_n)}(1-\omega_n).
$$
From \eqref{eq:cond_omega} it is easy to check that
\begin{align}
\left|c_{1,n}\right|&\le \frac{1}{4}\left|\omega_{n-1}-1\right| + (1+\Gamma/2)\left|\omega_n-1\right|,
\nonumber\\
\label{eq:cota_c2n}
\left|c_{2,n}\right|&\le (1+\Gamma)\left|\omega_n-1\right|,
\\
\left|c_{3,n}\right|&\le \Bigl(\frac{1}{4}+ \frac{1+\Gamma}{2}\Bigr)\left|\omega_n-1\right|,
\end{align}
so that 
\begin{equation}
\label{eq:cota_cn}
\max_{1\le j\le 3} \left|c_{j,n}\right|\le \frac{1}{4}\left|\omega_{n-1}-1\right| + (1+\Gamma)\left|\omega_n-1\right|.
\end{equation}
Furthermore, since
$$
G_n=\Bigl( 1+\frac{\omega_{n-1}}{2(1+\omega_{n-1})}\Bigr) \left\| \be_h^{n}\right\|_0^2 - (\be_h^{n},\be_h^{n-1}) +\frac{1}{4}
\left\| \be_h^{n-1}\right\|_0^2,
$$
and the smallest eigenvalue of matrix
\begin{equation}
\left[\begin{array}{cc} 1+x & -1/2\\ -1/2 & 1/4\end{array}\right]
\label{eq:una_matriz}
\end{equation}
is
$\lambda=\left((5+4x)-\sqrt{(5+4x)^2-16x}\right)/8$, which can be seen to be $\lambda\ge 0.14x$,
for $0\le x\le 1/2$ (see Fig.~\ref{fig:comparo}). Then,
if follows that
\begin{equation}
G_{n} \ge \frac{0.07}{(1+\Gamma)}{\omega_{n-1}} \bigl(\left\|\be_h^{n-1}\right\|_0^2 + \left\|\be_h^{n}\right\|_0^2\bigr).
\label{eq:cota_inf_Gn}
\end{equation}
Thus, if~\eqref{eq:cond_omega} holds, from~\eqref{eq:cota_cn} and~\eqref{eq:cota_inf_Gn} it follows that
\begin{eqnarray}
\label{eq:cota_calEn}
\left| {\cal R}_n\right| &\le& K_0\frac{1+\Gamma}{4\gamma} \left(\left|\omega_{n-1}-1\right| + 4(1+\Gamma)\left|\omega_n-1\right|
\right)G_n \nonumber\\
&\le& \frac{K_0}{\gamma}(1+\Gamma)^2 \left(\left|\omega_{n-1}-1\right| + \left|\omega_n-1\right|\right)G_n,
\end{eqnarray}
with $K_0\le (3/8)100/7=75/14$.
\begin{figure}
\begin{center}
\includegraphics[height=2.7truecm]{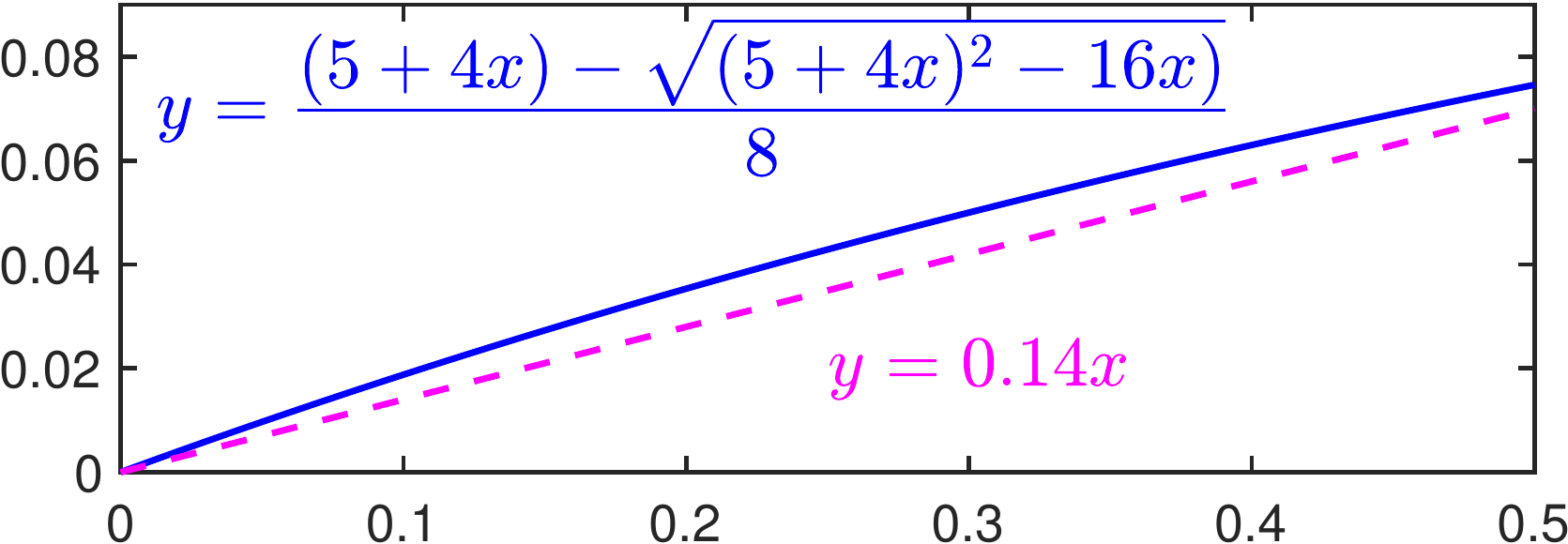}
\end{center}
\caption{Graphs of smallest eigenvalue of matrix~\eqref{eq:una_matriz} and $y=0.14x$.}
\label{fig:comparo}
\end{figure}
Inserting \eqref{eq:cota_calEn} into \eqref{eq:pvar1} we finally obtain \eqref{eq:pvar1_d}.
\end{proof}
In the sequel, we define
$$
\Delta t=\max_{1\le n\le N}\Delta t_n.
$$
Also, for $\alpha\in[1,10]$, we consider the polynomial
\begin{equation}\label{poly}
\frac{x(1+2x)}{\alpha}x^2+\frac{(1+2x)(1+x)}{\alpha}x-17.
\end{equation}
One can check that it has two complex roots and two real ones, one negative and one positive.  Let $\Gamma^*(\alpha)$ be this positive real root. It is possible to check that $\Gamma^*(
\alpha)$ is an increasing function of~$\alpha$ whose vales for $\alpha=1,5,10$ are larger than, $1.25$, $2.12$ and~$2.61$, respectively. Fixed $\alpha\in[1,10]$, we restrict ourselves to meshes satisfying
\begin{equation}
\label{cota_Gamma}
\Gamma\le\Gamma^*(\alpha).
\end{equation}
We also assume
\begin{equation}
\label{ratios_Lambda}
\sum_{n=2}^{N-1}\left(\left|\omega_{n-1}-1\right|
 + \left|\omega_n-1\right|\right)\le \Lambda,
\end{equation}
for some positive $\Lambda$.

We now state and prove the convergence result for the IMEX method~\eqref{fully_discrete}.
\begin{Theorem}\label{th:th2}
  Fix $\kappa>0$, and $\sigma>1$,
and let $\Delta t_n$ and~$h$ satisfy, the following CFL-type condition
\begin{eqnarray}\label{CFL2-u}
3c_{\rm inv}^2\|\bu\|_{L^\infty(L^\infty)}^2(1+2\Gamma)^2\frac{(1+\Gamma)}{\gamma}\frac{\Delta t_n}{h^2}\le \frac{\kappa}{T}, \quad n\ge 1,
\end{eqnarray}
and
\begin{eqnarray}\label{mas_hip}
1-\Delta t_n\frac{(1+\Gamma)L_h^n}{0.07\omega_n}&>&\frac{1}{\sigma}, \quad n\ge 1.
\end{eqnarray}
where $L_h^n$ is defined in \eqref{eq:Ln}. Let $G_n$ be as defined in~\eqref{eq:elG_n} for the error~$\be_h^n$ of method~\eqref{fully_discrete}.
Then, the following bound holds with $m_k$ defined in \eqref{mas_hip2}
\begin{eqnarray}\label{cota_def}
&&G_{n+1}+\nu\sum_{k=2}^{n+1}\Delta t_{k-1}\|\nabla e_h^k\|_0^2+\mu\sum_{k=2}^{n+1}\Delta t_k m_k\|\nabla \cdot \be_h^k\|_0^2\le\\
&&
C_n
\biggl((C_0^2+TC_3^2)(\Delta t)^4 + T\left(C_1^2h^{2k} + C_2^2h^{2l+2}\right)
\biggr),\nonumber
\end{eqnarray}
where
\begin{equation}\label{laCn}
C_n=\exp\left(\sigma_{n+1}\Delta t_{n}\frac{(1+\Gamma)L_h^{n}}{0.07\omega_{n}}+\sum_{k=2}^{n}
\left(\sigma_k\Delta t_{k-1}\frac{(1+\Gamma)L_h^{k-1}}{0.07\omega_{k-1}}+\omega_k f_k\right)\right),
\end{equation}
and the constants
$f_k$ and~$\sigma_k$ are defined in \eqref{laf}, \eqref{la_sigma}, respectively,
and the constants $C_i^2$ are defined in \eqref{eq:C2} and~\eqref{eq:C0}. 
\end{Theorem}
\begin{proof}
Taking $\bv_h=\be_h^{n+1}$ in \eqref{fully_error} and applying \eqref{eq:pvar1_d} from Lemma\, \ref{le:tech} we reach
\begin{eqnarray}\label{er_var_1}
G_{n+1} - G_n+\frac{\omega_n}{2(1+\omega_n)}\left\|  \be_h^{n+1}-\hat\be_h^n\right\|_0^2+\nu\Delta t_n\|\nabla \be_h^{n+1}\|_0^2
+\mu\Delta t_n\|\nabla \cdot\be_h^{n+1}\|_0^2\le\qquad\\
\Delta t_n|b(\hat \bs_h^n,\hat\bs_h^n,\be_h^{n+1})-b(\hat\bu_h^n,\hat \bu_h^n,\be_h^{n+1})|
+\Delta t_n(\tau_1^{n+1}+\tau_3^{n+1},\be_h^{n+1})+\Delta t_n(\tau_2^{n+1},\nabla \cdot\be_h^{n+1})\nonumber\\
+\frac{K_0}{\gamma}(1+\Gamma)^2 \left(\left|\omega_{n-1}-1\right| + \left|\omega_n-1\right|\right)G_n.\nonumber
\end{eqnarray}
As before, we now apply Lemma~\ref{le:nl} to estimate the nonlinear term and, further, we bound
\begin{align}
\label{eq:cota_ref_2}
2\|\hat \bu_h^n\|_\infty c_{\rm inv} h^{-1}\|\be_h^{n+1}-\hat e_h^n\|_0\|\be_h^{n+1}\|_0
\le& 2c_{\rm inv}^2\|\hat \bu_h^n\|_\infty^2\frac{(1+\omega_n)}{\omega_n}\frac{\Delta t_n}{h^2}\|\be_h^{n+1}\|_0^2
\\ &{}+
\frac{\omega_n}{2(1+\omega_n)}\frac{\|\be_h^{n+1}-\hat e_h^n\|_0^2}{\Delta t_n}.
\nonumber
\end{align}
Instead of \eqref{CFL2-u} we will now assume
\begin{eqnarray}\label{CFL2}
2c_{\rm inv}^2\|\hat \bu_h^n\|_\infty^2\frac{(1+\omega_n)}{\omega_n}\frac{\Delta t_n}{h^2}\le \frac{\kappa}{T}.
\end{eqnarray}
As before, we will see at the end of the proof that~\eqref{CFL2} holds if~\eqref{CFL2-u} holds and $h$ is sufficiently small.
Thus, from \eqref{eq:nonli1-0_var} in Lemma~\ref{le:nl}, \eqref{eq:cota_ref_2} and assuming \eqref{CFL2} we obtain
\begin{eqnarray}\label{eq:nonli_var}
&&\Delta t_n|b(\hat \bs_h^n,\hat \bs_h^n,\be_h)-b(\hat\bu_h^n,\hat\bu_h^n,\be_h)|\le
\Delta t_n\left(\frac{\|\nabla \hat \bs_h^n\|_\infty}{2}+\frac{9\alpha}{2\mu}\|\hat \bs_h^n\|_\infty^2+\frac{\kappa}{T}\right)\|\be_h^{n+1}\|_0^2\nonumber\\
&&\quad +\Delta t_n\frac{(1+2\Gamma)(1+\Gamma)}{2}\|\nabla \hat \bs_h^n\|_\infty\|\be_h^n\|_0^2+\Delta t_n\frac{\Gamma(1+2\Gamma)}{2}\|\nabla \hat \bs_h^n\|_\infty\|\be_h^{n-1}\|_0^2
\nonumber\\&&\quad+\Delta t_n\frac{\mu(1+2\Gamma)(1+\Gamma)}{18\alpha}\|\nabla \cdot\be_h^n\|_0^2+\Delta t_n\frac{\mu\Gamma(1+2\Gamma)}{18\alpha}\|\nabla \cdot\be_h^{n-1}\|_0^2
\\
&&\quad
+\frac{\omega_n}{2(1+\omega_n)}{\|\be_h^{n+1}-\hat \be_h^n\|_0^2}.\nonumber
\end{eqnarray}
Inserting \eqref{eq:nonli_var} into \eqref{er_var_1} and, arguing as before with the terms involving $\tau_1^{n+1}+\tau_3^{n+1}$ and~$\tau_2^{n+1}$, we get
\begin{eqnarray}\label{nonli_var2}
&&G_{n+1} - G_n+\nu\Delta t_n\|\nabla \be_h^{n+1}\|_0^2
+\frac{17}{18}\mu\Delta t_n\|\nabla \cdot\be_h^{n+1}\|_0^2\le\nonumber\\
&& \quad\Delta t_n\left(\frac{\|\nabla \hat \bs_h^n\|_\infty}{2}+\frac{9\alpha}{2\mu}\|\hat \bs_h^n\|_\infty^2
+\frac{\kappa+1}{T}\right)\|\be_h^{n+1}\|_0^2
+\frac{(1+2\Gamma)(1+\Gamma)}{2}\Delta t_n\|\nabla \hat \bs_h^n\|_\infty\|\be_h^n\|_0^2
\nonumber\\
&&\quad+\frac{\Gamma(1+2\Gamma)}{2}\Delta t_n\|\nabla \hat \bs_h^n\|_\infty\|\be_h^{n-1}\|_0^2+\Delta t_n\frac{\mu(1+2\Gamma)(1+\Gamma)}{18\alpha}\|\nabla \cdot\be_h^n\|_0^2\nonumber\\
&&\quad+\Delta t_n\frac{\mu\Gamma(1+2\Gamma)}{18\alpha}\|\nabla \cdot\be_h^{n-1}\|_0^2+\Delta t_n\frac{T}{2}\|\tau_1^{n+1}\|_0^2\nonumber\\
&&\quad+\Delta t_n\frac{T}{2}\|\tau_3^{n+1}\|_0^2
+\frac{9\Delta t_n}{2\mu}\|\tau_2^{n+1}\|_0^2+\frac{K_0}{\gamma}(1+\Gamma)^2 \left(\left|\omega_{n-1}-1\right| + \left|\omega_n-1\right|\right)G_n.
\end{eqnarray}
Let us denote by
\begin{equation}
\label{eq:Ln}
L_h^n=\frac{\|\nabla \hat \bs_h^n\|_\infty}{2}+\frac{9\alpha}{2\mu}\|\hat \bs_h^n\|_\infty^2
+\frac{\kappa+1}{T},
\end{equation}
so that, using \eqref{eq:cota_inf_Gn}, the first term on the right-hand side of \eqref{nonli_var2} can be written as
$$
\Delta t_n L_h^n\|\be_h^{n+1}\|_0^2\le \Delta t_n\frac{(1+\Gamma)L_h^n}{0.07\omega_n}G_{n+1}.
$$
Similarly, for the terms involving $\|\be_h^n\|_0^2$ and~$\|\be_h^{n-1}\|_0^2$ on the right-hand side
in~\eqref{nonli_var2}, we have
\begin{eqnarray*}
\frac{(1+2\Gamma)(1+\Gamma)}{2}\Delta t_n\|\nabla \hat \bs_h^n\|_\infty\|\be_h^n\|_0^2
+\frac{\Gamma(1+2\Gamma)}{2}\Delta t_n\|\nabla \hat \bs_h^n\|_\infty\|\be_h^{n-1}\|_0^2\le
\nonumber\\
 \frac{(1+\Gamma)(1+2\Gamma)^2}{0.14\omega_{n-1}}\Delta t_n\|\nabla \hat \bs_h^n\|_\infty G_n.
\end{eqnarray*}
Thus, going back to \eqref{nonli_var2} we reach
\begin{eqnarray}\label{nonli_var3}
&&G_{n+1} - G_n+\nu\Delta t_n\|\nabla \be_h^{n+1}\|_0^2
+\frac{17}{18}\mu\Delta t_n\|\nabla \cdot\be_h^{n+1}\|_0^2\le\Delta t_n\frac{(1+\Gamma)L_h^n}{0.07\omega_n}G_{n+1}\nonumber\\
&&\quad+\left(\frac{(1+\Gamma)(1+2\Gamma)^2}{0.14\omega_{n-1}}\Delta t_n\|\nabla \hat \bs_h^n\|_\infty+\frac{K_0}{\gamma}(1+\Gamma)^2 \left(\left|\omega_{n-1}-1\right| + \left|\omega_n-1\right|\right)\right) G_n\nonumber\\
&&\quad+\Delta t_n\frac{\mu(1+2\Gamma)(1+\Gamma)}{18\alpha}\|\nabla \cdot\be_h^n\|_0^2+\Delta t_n\frac{\mu\Gamma(1+2\Gamma)}{18\alpha}\|\nabla \cdot\be_h^{n-1}\|_0^2 \\
&&\quad+\frac{\Delta t_n}{2}\left(T(\|\tau_1^{n+1}\|_0^2+\|\tau_3^{n+1}\|_0^2)
+\frac{9}{\mu}\|\tau_2^{n+1}\|_0^2\right).\nonumber
\end{eqnarray}
Denoting by
\begin{equation}\label{laf}
f_n=\frac{(1+\Gamma)(1+2\Gamma)^2}{0.14\omega_{n-1}}\Delta t_n\|\nabla \hat \bs_h^n\|_\infty+\frac{K_0}{\gamma}(1+\Gamma)^2 \left(\left|\omega_{n-1}-1\right|
 + \left|\omega_n-1\right|\right),
\end{equation}
and adding terms we get
\begin{eqnarray}\label{nonli_var4}
&&G_{n+1}+\nu\sum_{k=2}^{n+1}\Delta t_{k-1}\|\nabla \be_h^k\|_0^2\nonumber\\
&&\quad +\mu\sum_{k=2}^{n+1}\Delta t_k\left(\frac{17}{18}\frac{1}{\omega_k}
-\frac{(1+2\Gamma)(1+\Gamma)}{18\alpha}-\frac{\Gamma(1+2\Gamma)\omega_{k+1}}{18\alpha}\right)\|\nabla \cdot \be_h^k\|_0^2\le\nonumber\\
&&G_1+\mu\left(\frac{(1+2\Gamma)(1+\Gamma)\Delta t_1}{18\alpha}+\frac{\Gamma(1+2\Gamma)\Delta t_2}{18\alpha}\right)\|\nabla \cdot \be_h^1\|_0^2+\mu\frac{\Gamma(1+2\Gamma)\Delta t_1}{18\alpha}\|\nabla \cdot \be_h^0\|_0^2\nonumber\\
&&\quad+\sum_{k=2}^{n+1}\Delta t_{k-1}\frac{(1+\Gamma)L_h^{k-1}}{0.07\omega_{k-1}}G_{k}
+\sum_{k=2}^{n+1} f_{k-1} G_{k-1}\nonumber\\
&&\quad+\frac{1}{2}\sum_{k=2}^{n+1}\Delta t_{k-1}\left(T(\|\tau_1^k\|_0^2+\|\tau_3^k\|_0^2)
+\frac{9}{\mu}\|\tau_2^k\|_0^2\right).
\end{eqnarray}
As mentioned before, for any fixed value of $\alpha$, the number
$\Gamma^*(\alpha)$ in~\eqref{cota_Gamma} is chosen to be the 
the positive real root of \eqref{poly}.
Since  $\omega_k\le \Gamma^*$, $k=1,\ldots,N-1$, it follows  that
\begin{eqnarray}\label{mas_hip2}
m_k:= \frac{17}{18}\frac{1}{\omega_k}
-\frac{(1+2\Gamma)(1+\Gamma)}{18\alpha}-\frac{\Gamma(1+2\Gamma)\omega_{k+1}}{18\alpha}&\ge& 0.
\end{eqnarray}
Let us also denote by
\begin{eqnarray}\label{ini_var}
E_0^v&=&(1+f_1)G_1+\mu\left(\frac{(1+2\Gamma)(1+\Gamma)\Delta t_1}{18\alpha}+\frac{\Gamma(1+2\Gamma)\Delta t_2}{18\alpha}\right)\|\nabla \cdot \be_h^1\|_0^2\nonumber\\
&&\quad+\mu\frac{(1+2\Gamma)(1+\Gamma)\Delta t_1}{18\alpha}\|\nabla \cdot \be_h^0\|_0^2.
\end{eqnarray}
Taking into account \eqref{ini_var},
\eqref{nonli_var4} can be written as
\begin{eqnarray}\label{nonli_var5}
&&G_{n+1}+\nu\sum_{k=2}^{n+1}\Delta t_{k-1}\|\nabla \be_h^k\|_0^2+
\mu\sum_{k=2}^{n+1}\Delta t_k m_k\|\nabla \cdot \be_h^k\|_0^2\le
\\
&&\quad\quad 
\Delta t_n\frac{(1+\Gamma)L_h^n}{0.07\omega_n} G_{n+1}
+\sum_{k=2}^{n}\left(\Delta t_{k-1}\frac{(1+\Gamma)L_h^{k-1}}{0.07\omega_{k-1}}+\omega_k f_k\right)G_{k}
\nonumber\\
&&\quad+\frac{1}{2}\sum_{k=2}^{n+1}\Delta t_{k-1}\left(T(\|\tau_1^k\|_0^2+\|\tau_3^k\|_0^2)
+\frac{9}{\mu}\|\tau_2^k\|_0^2\right) + E_0^v.\nonumber
\end{eqnarray}
Now assumption~\eqref{mas_hip} implies that $\sigma_{n+1}$ defined by
\begin{equation}\label{la_sigma}
\sigma_{n+1}=\biggl(1-\Delta t_n\frac{(1+\Gamma)L_h^n}{0.07\omega_n} \biggr)^{-1},
\end{equation}
is positive. Thus,
applying Lemma~\ref{gronwall2} we have
\begin{eqnarray}\label{nonli_var6}
&&G_{n+1}+\nu\sum_{k=2}^{n+1}\Delta t_{k-1}\|\nabla e_h^k\|_0^2+\mu\sum_{k=2}^{n+1}\Delta t_k m_k\|\nabla \cdot \be_h^k\|_0^2\le\\
&&
C_n
\biggl(E_0^v+
\frac{1}{2}\sum_{k=2}^{n+1}\Delta t_{k-1}\left(T(\|\tau_1^k\|_0^2+\|\tau_3^k\|_0^2)
+\frac{9}{\mu}\|\tau_2^k\|_0^2\right)\biggr),\nonumber
\end{eqnarray}
where $C_n$ is defined in \eqref{laCn}.

To conclude we bound the initial and truncation errors. For the initial errors, we argue
as in the fixed stepsize case. Since we assume $\bu_h^0=\bs_h^0$ so that $\be_h^0=0$. Consequently,
$$
G_1=\frac{\omega_0}{2(1+\omega_0)} \left\|\be_h^1\right\|_0^2 + \frac{1}{4}\left\|2\be_h^1\right\|_0^2 \le
\frac{3}{2} \left\|\be_h^1\right\|_0^2 ,
$$
and, in view of \eqref{ini_var},
\begin{equation}
\label{ini_var2}
E_0^v \le (1+f_1)\frac{3}{2} \left\|\be_h^1\right\|_0^2 + \omega_0\left(\frac{(1+2\Gamma)(1+\Gamma)}{18\alpha} + \frac{\Gamma(1+2\Gamma)\omega_1}{18\alpha}\right)
\mu\Delta t_0 \left\|\nabla\cdot\be_h^1\right\|_0^2.
\end{equation}

 The
first step is carried out with the implicit-explicit Euler method and then, in view of \eqref{ini_var2}, and arguing as we did with~\eqref{cota_E0} we get
\begin{equation}\label{cota_E0v}
E_0^v \le  C\left(1+\Delta t\left\|\bu\right\|_{L^\infty(W^{1,\infty})} \right)
\left(C_0^2(\Delta t)^4 + T\left(C_1^2h^{2k} + C_2^2h^{2l+2}\right)\right),
\end{equation}
where $C$ depends on~$\Gamma$, $1/\gamma$, $\omega_0$ and~$\omega_1$.

For the truncation errors we also argue as in the fixed stepsize case. Then, arguing exactly as in \eqref{eq:suma_truncs} we get
\begin{eqnarray}\label{trunca_vari}
\frac{1}{2}\sum_{k=2}^{n+1}\Delta t_{k-1}\left(T(\|\tau_1^k\|_0^2+\|\tau_3^k\|_0^2)
+\frac{9}{\mu}\|\tau_2^k\|_0^2\right)\le T\left(C_1^2h^{2k} + C_2^2h^{2l+2} + C_3^2(\Delta t)^4\right),
\end{eqnarray}
with the constants $C_i^2$ defined in \eqref{eq:C2} and avoiding writing the explicit dependency on $\omega_n$.
Inserting  \eqref{cota_E0v} and \eqref{trunca_vari} into \eqref{nonli_var6} we finally reach \eqref{cota_def}.

To conclude the proof, as in the fixed step size, we now show that taking $h$ sufficiently small, the CFL type condition~\eqref{CFL2-u} implies~\eqref{CFL2}. Arguing as in the fixed step-size, it is possible to choose $h_{0,1}$ such that $\left\|\bs_h^n - \bu^n\right\|_\infty \le \left\|\bu\right\|_{L^\infty(L^\infty)} /12$ for $h\le h_{0,1}$, which implies
\begin{equation}
\label{new_bsh}
\left\| \bs_h^n\right\|_\infty \le \frac{13}{12}\left\|\bu\right\|_{L^\infty(L^\infty)},
\end{equation}
and, since
$\left\|\hat\bs_h^n\right\| \le (1+2\Gamma)\left\|\bs_h\right\|_{L^\infty(L^\infty)}$, by writing $\hat\bu_h^n=(\hat\bu_h^n-\hat\bs_h^n) + \hat\bs_h^n$, it also holds
\begin{equation}
\left\|\hat\bu_h^n\right\|_\infty \le c_{\rm inv}h^{-d/2}\left((1+\Gamma) \left\|\bs_h^n-\bu_h^n \right\|_0 + \Gamma
 \left\|\bs_h^{n-1}-\bu_h^{n-1} \right\|_0\right) + (1+2\Gamma) 
 \frac{13}{12}\left\|\bu\right\|_{L^\infty(L^\infty)}.
\label{eq:ulti}
\end{equation}
If an $h_{0,2}>0$ exists such that for $h\le h_{0,2}$ both the right-hand side of~\eqref{cota_ini} and the right-hand side of~\eqref{cota_def}
multiplied by $0.07(1+\Gamma)/\gamma$
are smaller than $h^d \left\|\bu\right\|_{L^\infty(L^\infty)}^2/(12c_{\rm inv})^2$ it follows
that 
$$
\|\bs_h^k-\bu_h^k\|_0\le \frac{h^{d/2}}{12c_{\rm inv}} \|\bu\|_{L^\infty(L^\infty)},\quad k=1,\cdots,n.
$$
Thus, in view of~\eqref{eq:ulti}, it follows that
$$
\left\|\hat\bu_h^n\right\|_\infty \le (1+2\Gamma) \frac{7}{6} \left\|\bu\right\|_{L^\infty(L^\infty)},
$$
so that \eqref{CFL2} holds as a consequence of assumption~\eqref{CFL2-u}. We now argue why $h_{0,2}$ exists, and this will finish the proof. The exponents of the powers of~$h$ on right-hand sides of~\eqref{cota_ini} and~\eqref{cota_def} are larger than $d$, so that they will decay with $h$ faster than $h^d$ if the exponential term
in~\eqref{laCn} remains bounded as $h\rightarrow 0$. To show that this is the case, we first notice that,
in view of the expression of~$L_h^n$ in~\eqref{eq:Ln} and recalling~\eqref{cotainfty1} and~\eqref{new_bsh}, we have
$$
L_h^n \le (1+2\Gamma) \left(\frac{(1+C_\infty)}{2}\left\|\bu\right\|_{L^\infty(W^{1,\infty})} + \frac{9\alpha}{2\mu}\left(\frac{13}{12}\right)^2
\left\|\bu\right\|_{L^\infty(L^\infty)}^2\right) + \frac{\kappa+1}{T}
$$
and in view of the expression of~$f_k$ in~\eqref{laf}, we also have
$$
f_k \le \frac{(1+\Gamma)(1+2\Gamma)^3}{0.14\gamma}\Delta t_n(1+C_\infty)\left\|\bu\right\|_{L^\infty(W^{1,\infty})}+\frac{K_0}{\gamma}(1+\Gamma)^2 \left(\left|\omega_{n-1}-1\right|
 + \left|\omega_n-1\right|\right).
 $$
Finally, due to assumption~\eqref{mas_hip}, $\sigma_n$ defined in~\eqref{la_sigma} satisfies $\sigma_n<\sigma$ and noticing that the sum of the increments $\Delta t_n$ is always bounded by~$T$ and recalling~assumption~\eqref{ratios_Lambda}, it is clear that the constants $C_n$ in~\eqref{laCn} remain bounded as $h\rightarrow0$.
\end{proof}

\begin{remark}
In view of \eqref{CFL2-u} we notice that the ratios $\Delta t_n/h^2$ must remain bounded. 
As stated in Remark \ref{re:2}, the size of the constant $\kappa$ in the CFL condition \eqref{CFL2-u}
reflects in the error bound \eqref{cota_def} through an exponential factor.
More precisely, considering also the size of the constant $\sigma$ in \eqref{mas_hip}, taking into account $\sigma_n<\sigma$ (with $\sigma_n$ defined in \eqref{la_sigma}) and in view of the expression of~$L_h^n$ in~\eqref{eq:Ln} the influence of the constants $\kappa$ and $\sigma$ 
 affect the constant $C_n$ in~\eqref{nonli_var6} by a factor of size
$$
\exp\biggl(\kappa \sigma\frac{1+\Gamma}{0.07T} \sum_{k=1}^{n} \frac{\Delta t_k}{\omega_{k}} \biggr)
=\exp\biggl(\kappa \sigma\frac{1+\Gamma}{0.07T} \sum_{k=1}^{n}\Delta t_{k-1} \biggr)
\le \exp\biggl(\kappa \sigma\frac{1+\Gamma}{0.07}\biggr).
$$
On the other hand, the value of $\Gamma^*(\alpha)$ in \eqref{cota_Gamma} increases with the value of the free parameter $\alpha$. However,
arguing as above, we can observe that
the effect of increasing $\alpha$ (to allow for larger increases in step length) also affects the constant $C_n$ in~\eqref{laCn} since the term $\frac{9\alpha}{2\mu}\|\hat \bs_h^n\|_\infty^2$ is present in the definition of~$L_h^n$ in~\eqref{eq:Ln}.
\end{remark}
\bigskip


\begin{Theorem}\label{th:th2-si}
  Fix $\sigma>1$,
and let $\Delta t_n$ and~$h$ satisfy
$$
1-\Delta t_n\frac{(1+\Gamma){\cal L}_h^n}{0.07\omega_n}>\frac{1}{\sigma}, \quad n\ge 1.
$$
where
$$
{\cal L}_h^n=\frac{\|\nabla \bs_h^{n+1}\|_\infty}{2}+\frac{9\alpha}{2\mu}\| \bs_h^{n+1}\|_\infty^2
+\frac{1}{T},
$$
Let $G_n$ be as defined in~\eqref{eq:elG_n} for the error~$\be_h^n$ of the semi-implicit method~\eqref{semimp}.
Then, the following bound holds with $m_k$ defined in \eqref{mas_hip2}
\begin{eqnarray}\label{cota_def-si_semi}
&&G_{n+1}+\nu\sum_{k=2}^{n+1}\Delta t_{k-1}\|\nabla e_h^k\|_0^2+\mu\sum_{k=2}^{n+1}\Delta t_k m_k\|\nabla \cdot \be_h^k\|_0^2\le\\
&&
C_n
\biggl((C_0^2+TC_3^2)(\Delta t)^4 + T\left(C_1^2h^{2k} + C_2^2h^{2l+2}\right)
\biggr),\nonumber
\end{eqnarray}
where
\begin{equation*}
C_n=\exp\left(\sigma_{n+1}\Delta t_{n}\frac{(1+\Gamma){\cal L}_h^{n}}{0.07\omega_{n}}+\sum_{k=2}^{n}
\left(\sigma_k\Delta t_{k-1}\frac{(1+\Gamma){\cal L}_h^{k-1}}{0.07\omega_{k-1}}+\omega_k f_k\right)\right),
\end{equation*}
and the constants
$f_k$ and~$\sigma_k$ are defined in \eqref{laf}, \eqref{la_sigma}, respectively,
and the constants $C_i^2$ are defined in \eqref{eq:C2} and~\eqref{eq:C0}. 
\end{Theorem}

\begin{proof} The proof follows that of~Theorem~\ref{th:th2} with the changes that we now comment on. Indeed, it is easy to see that~\eqref{er_var_1}
also holds for the semi-implicit method~\eqref{semimp} with the first term on the right-hand side of~\eqref{er_var_1} replaced by
$$
|b(\hat \bs_h^n,\bs_h^{n+1},\be_h^{n+1})-b(\hat\bu_h^n, \bu_h^{n+1},\be_h^{n+1})|
$$
and~$\tau_3$ by~$\tau_4$. Then, recalling~Lemma~\ref{le:nl-si} and Remark~\ref{semimp-nl}, we have that 
\eqref{nonli_var2} holds with~$\kappa=0$, $\hat\bs_h^n$ replaced by~$\bs_h^{n+1}$ and~$\tau_3$ by~$\tau_4$. The proof is concluded following the steps of the proof of~Theorem~\ref{th:th2}, but without having to prove that the CFL condition~\eqref{CFL2} holds.
\end{proof} 

\begin{remark}\label{re_order_var} Analogous comments to those in Remark \ref{re_order} apply in this case taking into account 
\eqref{cota_def}, \eqref{cota_def-si_semi} 
and the definition of $G_n$ in \eqref{eq:elG_n}.
\end{remark}


\section{Numerical Experiments}
\label{Se:Numer}

We test methods~\eqref{fully_discrete} and~\eqref{semimp} with variable step size. We now describe the algorithm used, which follows standard procedures in the numerical integration of ordinary differential equations (ODEs).  For the implementation of the method, we used a variable formula  strategy as described in~\cite[\S~III.5]{Hairer} for the fully nonlinear BDF formulae, but with explicit treatment of the convective term as specified in~\eqref{fully_discrete}. In particular we notice that the local error estimate of the method of order~$k$ is given by
$$
{\rm EST}_n=\frac{\Delta t_n}{t_{n+1}-t_{n-k}} \left\|{\cal U}^{n,k}_h\right\|_0,
$$
where 
$$
{\cal U}^{n,k}_h= \prod_{i=0}^{k-1}(t_{n+1} -t_{n-i}) \bu_h[t_{n+1},\ldots,t_{n-k}],
$$
where $\bu_h[t_{n+1},\ldots,t_{n-k}]$ is the standard $(k+1)$-th divided difference based on $t_{n+1},\ldots,t_{n-k}$ and
$\bu_h^{n+1},\ldots, \bu_h^{n-k}$. This corresponds to the fully nonlinear BDF, but, nevertheless, it was used with the methods~(\ref{fully_discrete}) and~\eqref{semimp}.
At each time level~$t_n$, the estimation ${\rm EST}_n$ is compared with the quantity
$$
{\rm TOL}_n =TOL_r \left(\max\left\{\left\| \bu_h^{n+1}\right\|_0,\left\| \bu_h^{n}\right\|_0\right\}+0.001\right),
$$
where $TOL_r$ is a given tolerance.
If ${\rm EST_n}>{\rm TOL_n}$, $\bu_h^{n+1}$ is not considered sufficiently accurate and is rejected. It is then recomputed from $t_n$ with a new step length given by
\begin{equation}
\label{eq:change}
\Delta t_{n}^{\rm new} = 0.9\Delta t_n \sqrt[k+1]{\frac{{\rm TOL}_n}{{\rm EST}_n}}.
\end{equation}
On the contrary, if ${\rm EST_n}<{\rm TOL_n}$, then, $\bu_h^{n+1}$ is accepted and the algorithm proceeds to compute $\bu_h^{n+2}$ with a step length $\Delta t_{n+1}$ given by the right-hand side of~\eqref{eq:change}.

The algorithm starts with~$k=1$ and $\Delta t_0=\Delta t_1=\sqrt{TOL_r}/100$. At $t_1$ the first local error estimation~${\rm EST}_1$ is computed. If ${\rm EST}_1> {\rm TOL}_1$ the computation is restarted again from $t_0$ with the step size changed according to~\eqref{eq:change}. From $t_2$ onwards, estimates corresponding to the two-step BDF can be estimated. When these estimates are smaller than those corresponding to~$k=1$, then $k$ is switched to $k=2$. In the experiments below, no more than eight or ten steps were performed with $k=1$.

We now comment on the solution of the linear systems to be solved at each step, where, again, we follow standard procedures in the numerical integration of ODEs. If we denote by~$\by^n$ the vector containing the coefficients of the velocity~$\bu_h^{n+1}$ and the pressure
$p_h^{n+1}$ in the corresponding nodal basis of the finite element spaces, then, since~\eqref{fully_discrete} and~\eqref{semimp} are linear in~$\bu_h^{n+1}$ and~$p_h^{n+1}$, to find~$\by^n$ one has to solve a linear system
\begin{equation}
\label{eq:system}
A_n \by^n = \bb^n.
\end{equation}
This was solved by iterative refinement. Starting with an approximation~$\by^{n,[0]}= \hat \by^n$, extrapolated from the previous~$k$ values ($k=1,2$), the approximation is updated as
$
\by^{n,[j+1]} = \by^{n,[j]} + \be^{n,[j]}
$,
where $\be^{n,[j]}$ is obtained by solving the system~$A_m \be^{n,[j]}=\bb^n - A_n\by^{n,[j]} $, where, we notice, $A_m$ has been factored on time level~$t_m\le t_n$. The iteration is stopped either when the velocity in two consecutive iterations
satisfies
\begin{equation}
\label{eq:converge}
\left\| \bu_h^{n,[j+1]} - \bu_h^{n,[j]}\right\|_0 \le \min\left\{10^{-8},\frac{TOLr}{100} \right\}\left(\left\|\bu_h^{n-1}\right\|_0+ 0.001\right),
\end{equation}
in which case we set $\bu_n^{n}=\bu_n^{n,[j+1]}$, or if five iterations are performed without ever \eqref{eq:converge} being satisfied, in which case, $A_n$ is factored , replaces~$A_m$ and system~\eqref{eq:system} is solved. As we will see below, the number of factorizations never exceeded 5\% of the total number of steps.

The code was programmed in {\sc Matlab} and for the factorizations in linear systems we used command {\tt decomposition} with no extra arguments. The set of equations obtained from~(\ref{fully_discrete}) when $\bv_h=0$ was rewriten as
$$
-(\nabla\cdot \bu_h^{n+1},q_h)=0,
$$
so that matrices~$A_n$ were symmetric (although indefinite).

\subsection{Problem with known solution}


We consider the Navier-Stokes equations in  the domain $\Omega=[0,1]^2$ with $T=4$, and the forcing term~$\bff$ chosen so that the solution $\bu$ and~$p$ are given by
\begin{eqnarray}
\label{eq:exactau}
\bu(x,y,t)&=& \frac{6+4\cos(4t)}{10} \left[\begin{array}{c} 8\sin^2(\pi x) (2y(1-y)(1-2y)\\
-8\pi\sin(2*\pi x) (y(1-y))^2\end{array}\right]\\
p(x,y,t)&=&\frac{6+4\cos(4t)}{10} \sin(\pi x)\cos(\pi y).
\label{eq:exactap}
\end{eqnarray}

We present results with the $P_2/P_1$ pair of mixed finite-elements on a regular triangulation with SW-NE diagonals.
The meshes have $N=6$, $12$, $24$ and~$48$ subdivisions in each coordinate directions, and the tolerances for the BDF2 on these meshes are, respectively, $10^{-4}$, $10^{-5}$, $10^{-6}$ and~$10^{-7}$. We checked that they are sufficiently small so that the error arising from time discretization is negligible when compared to that arising form spatial discretization. In the experiments in this section, the value of the grad-div parameter is~$\mu=0.05$. 

In Fig.~\ref{fig:conv}, for the final time $t=4$, we show the $L^2$ errors in velocity
\begin{equation}
\label{errsu_en_linfty}
\|\bu_h^n-I_h(\bu^n)\|_0,
\end{equation}
where $I_h$ is the standard Lagrange interpolant. In Fig.~\ref{fig:conv}, as in the rest of the pictures in this section, results of IMEX method~\eqref{fully_discrete} are joined by continuous lines, while those of the semi-implict method~\eqref{semimp} by discontinuous lines. It can be seen that continuous and discontinuous lines are super imposed, reflecting the fact that the errors shown in Fig.~\ref{fig:conv} correspond to spatial discretization and are independent of the method chosen for time integration.
\begin{figure}
\begin{center}
\includegraphics[height=5.8truecm]{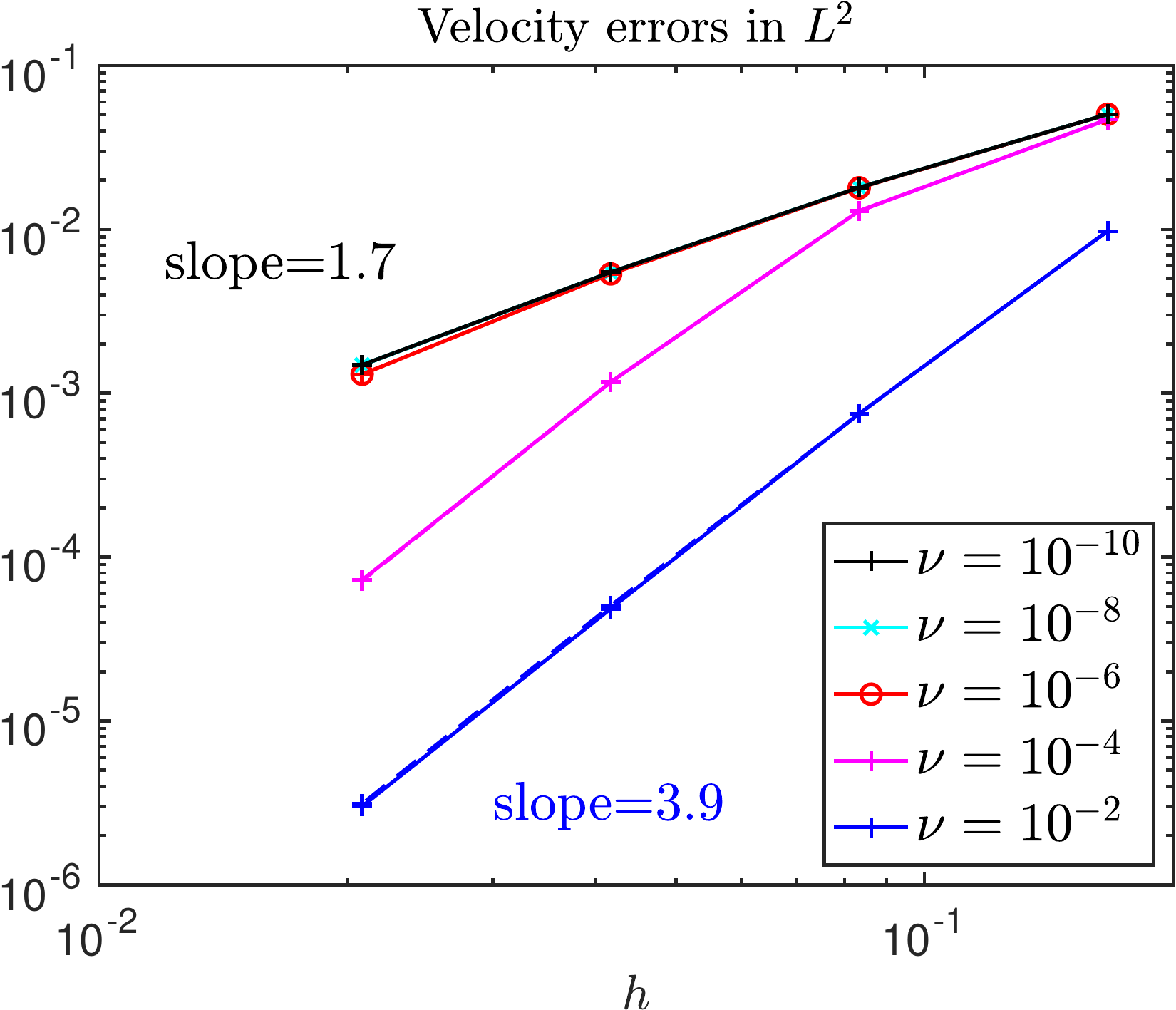}
\end{center}
\caption{Velocity errors~(\ref{errsu_en_linfty}) for $T=4$ and~$\mu=0.05$. Continuous lines correspond to method~(\eqref{fully_discrete} and continuous lines to~method~\eqref{semimp}.}
\label{fig:conv}
\end{figure}
It can be seen that, as $h$ decreases, the errors also decrease, but with a higher rate for the two largest values of the viscosity, $\nu=10^{-2}$ and~$\nu=10^{-4}$.  The results corresponding to~$\nu=10^{-6}$, $\nu=10^{-8}$ and $\nu=10^{-10}$ coincide and are on top of one another.
We also show the slopes of a least squares fit to the results corresponding to~$\nu=10^{-2}$ (in blue) and $\nu=10^{-10}$,
confirming the analysis that the rate of decay is $O(h^2)$ independent of~$\nu$ in the convection dominated regime, i.e., for $\nu$ small enough. For $\nu=10^{-2}$, and the last three results of~$\nu=10^{-4}$, the slope close to~4 reflects that the numerical approximation~$\bu_h$ is supraconvergent, in the sense that the error~\eqref{errsu_en_linfty} decays faster that $\|\bu^n-I_h(\bu^n)\|_0$, which is $O(h^3)$.

We check that our implementation uses variable step sizes in Fig.~\ref{fig:deltates}, where we show the step lengths of method~\eqref{fully_discrete} (the ones corresponding to~method~\eqref{semimp} being hardly distinguishable). They correspond  to $\nu=10^{-6}$, $TOL_r=10^{-5}$ and the mesh with $N=12$ subdivisions in each coordinate directions.  
After the initial steps corresponding to the first order method, the step lengths grow quickly, and then vary between~$0.001$ and~$0.005$ approximately.  In Fig.~\ref{fig:deltates} we marked with a red cross the locations where a step was rejected. We see that the number of rejections is really low. Most of the results corresponding to different values of~$\nu$, $TOL_r$ and~$N$ showed similar variations.
\begin{figure}
\begin{center}
\includegraphics[height=5.8truecm]{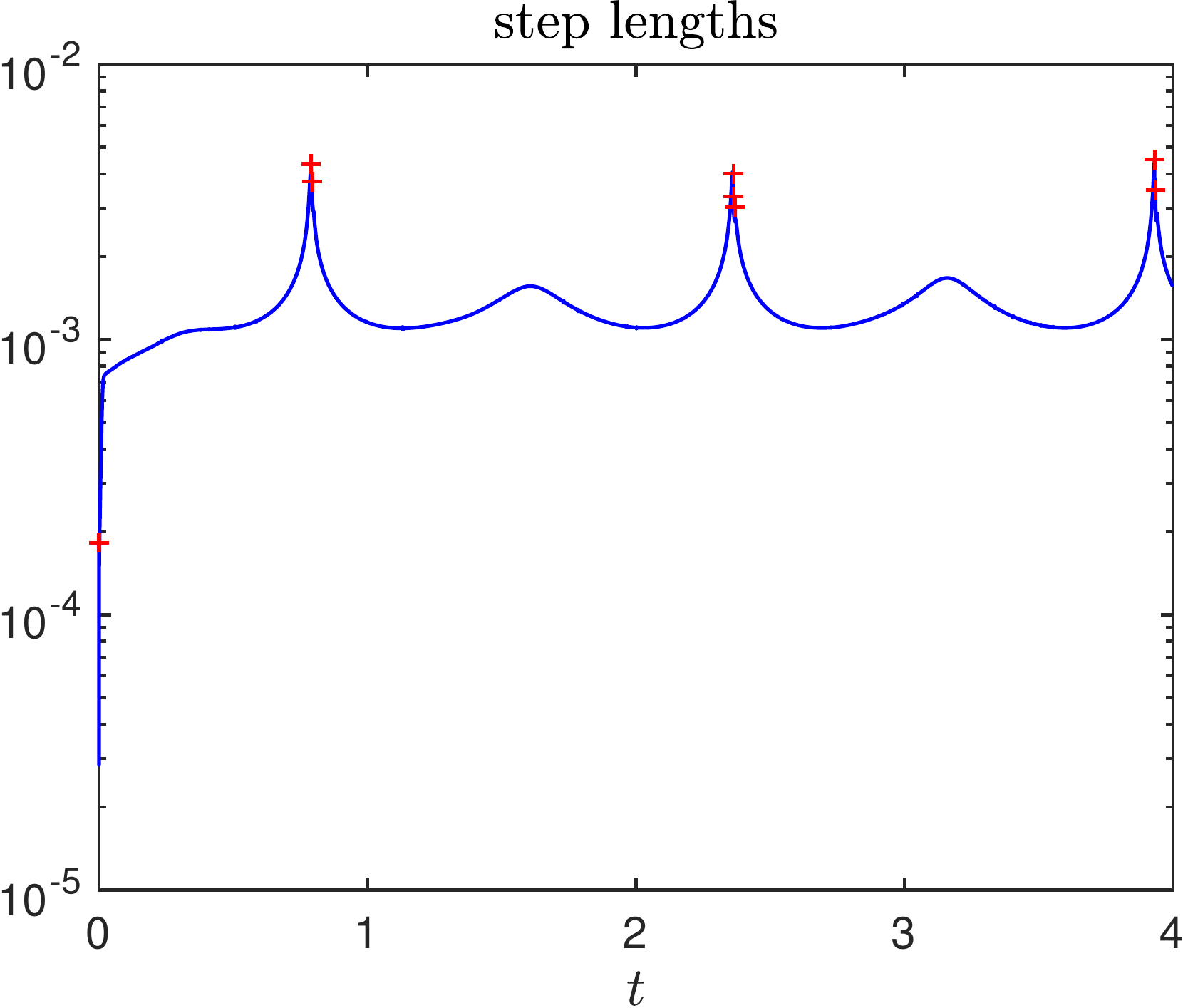}
\end{center}
\caption{Step lengths corresponding to $\nu=10^{-6}$, $TOL_r=10^{-5}$ and $N=12$. Rejections marked with a red cross.}
\label{fig:deltates}
\end{figure}

In Fig.~\ref{fig:steps-facts} we compare the number of factorizations with the number of steps for the results shown in Fig.~\ref{fig:conv} (recall each factorization is reused in subsequent steps for iterative refinement until failure to converge, where it is updated). We can see that the number of factorizations hardly grows as $h$ (and the corresponding tolerances) are decreased, and that at worst (large~$h$, small~$\nu$) they are hardly more than 5\% of the number of steps. Some differences can be seen between methods \eqref{fully_discrete} and~\eqref{semimp},  for large values of~$\nu$. We do not have an explanation for them. We notice, however, that the differences between the methods become less significant for smaller values of~$\nu$.

\begin{figure}
\begin{center}
\includegraphics[height=5.8truecm]{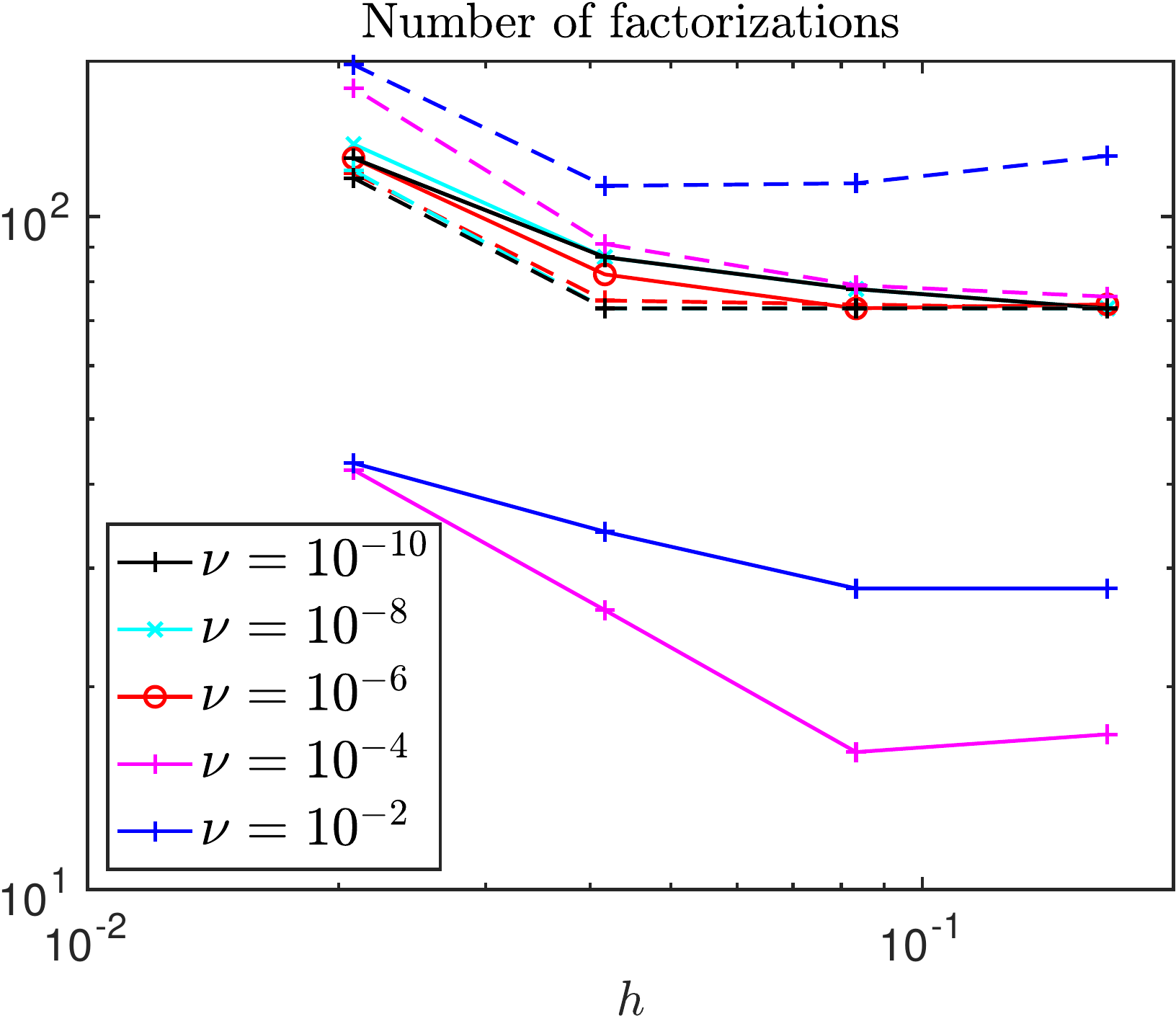} \quad 
\includegraphics[height=5.8truecm]{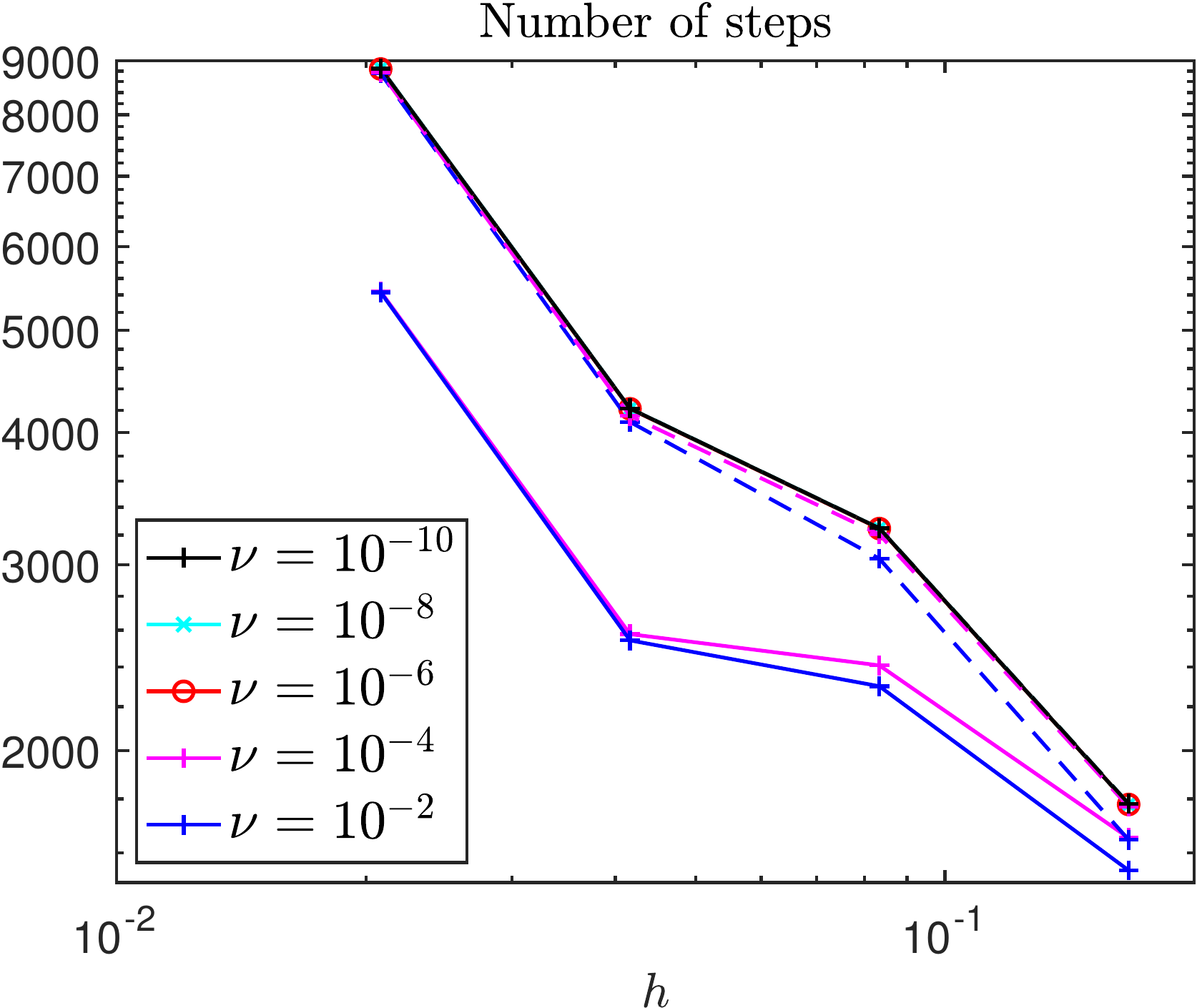}
\end{center}
\caption{Number of factorizations (left) and~steps (right) for results shown in~Fig.~\ref{fig:conv}.}
\label{fig:steps-facts}
\end{figure}

Next we check if the CFL-type condition~\eqref{CFL2-u} is necessary in practice for the IMEX method~(\ref{fully_discrete}), whereas this is not the case of the semi-implicit method~\eqref{semimp}. For that purpose, we repeat the above computations (and including also meshes with $N=96$ subdivisions in each coordinate direction) but with tolerance $TOL_r=10^{-4}$ for all meshes. Notice that in principle the number of steps depends on the tolerance, so, if the tolerance is kept constant while the spatial mesh size $h$ is decreased, the number of steps should not change much. On the other hand, if stability worsens as $h$ is decreased, errors arising from time discretization increase with every step, and will become larger than the tolerance more easily, forcing the algorithm to take smaller step sizes, this resulting in a larger number of steps. In Fig.~\ref{fig:steps-facts4} we show the number of steps (left plot). For the IMEX method (continuous lines), we see that although they remain fairly constant for $\nu=10^{-2}$, this is not the case of the rest of the values of the viscosity. Indeed, for the three smallest values of the viscosity, the number of steps roughly doubles from $h=1/24$ to~$h=1/48$, and it is multiplied by three from~$h=1/48$ to~$h=1/96$, approaching the factor~$4$ predicted by the CFL-type condition~\eqref{CFL2-u}. On the other hand, for the semi-implicit method (discontinuous lines) we see that, contrary to the IMEX method, the number of steps fairly constant as~$h$ is refined. As we have shown, this method does not require the CFL-type condition~\eqref{CFL2-u} to converge. On the contrary, in view of~Fig.~\ref{fig:steps-facts4}, we may conclude that method with explicit convection~(\ref{fully_discrete}) indeed suffers from restriction~\eqref{CFL2-u}.

 \begin{figure}
\begin{center}
\includegraphics[height=5.8truecm]{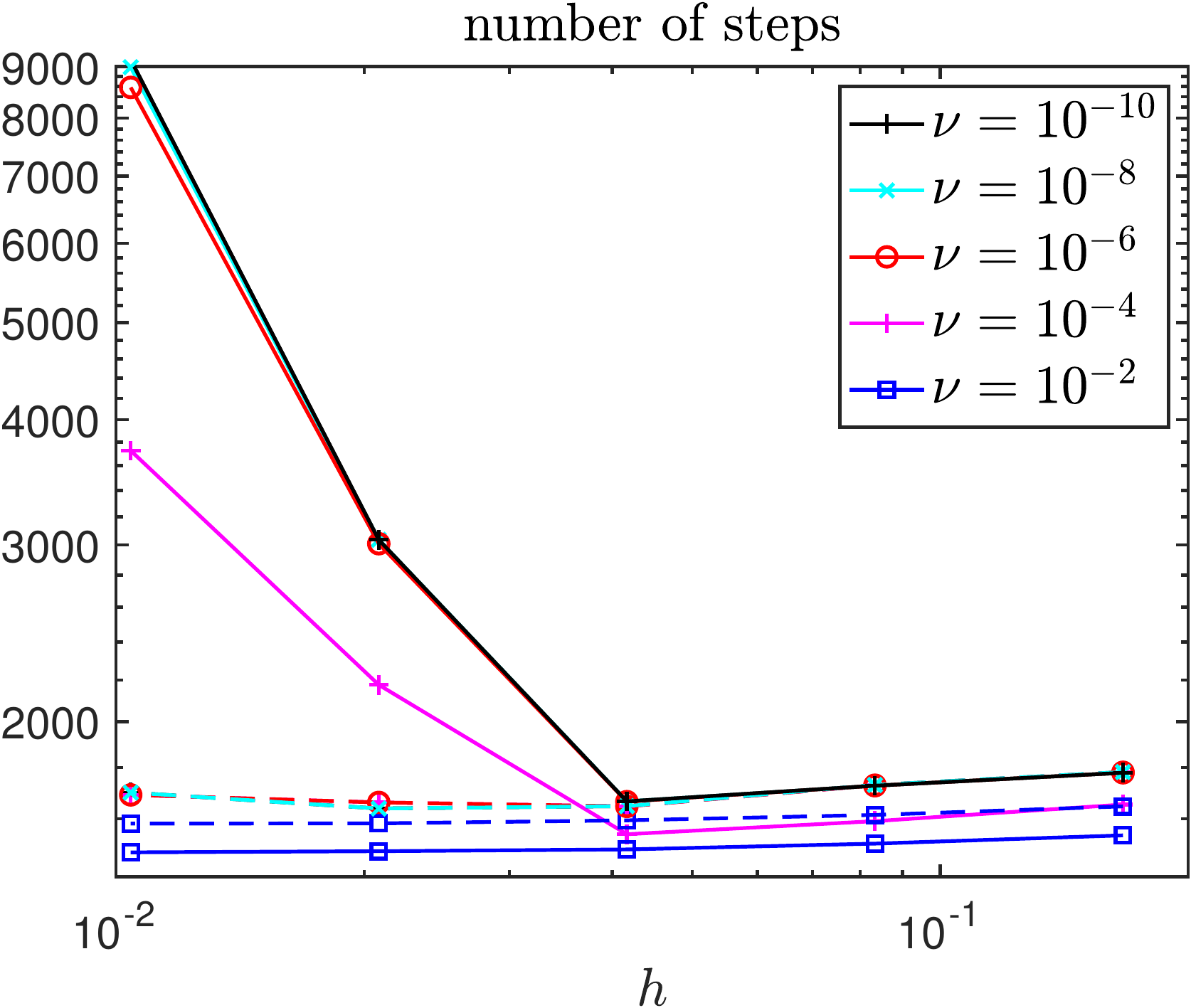} \quad 
\includegraphics[height=5.8truecm]{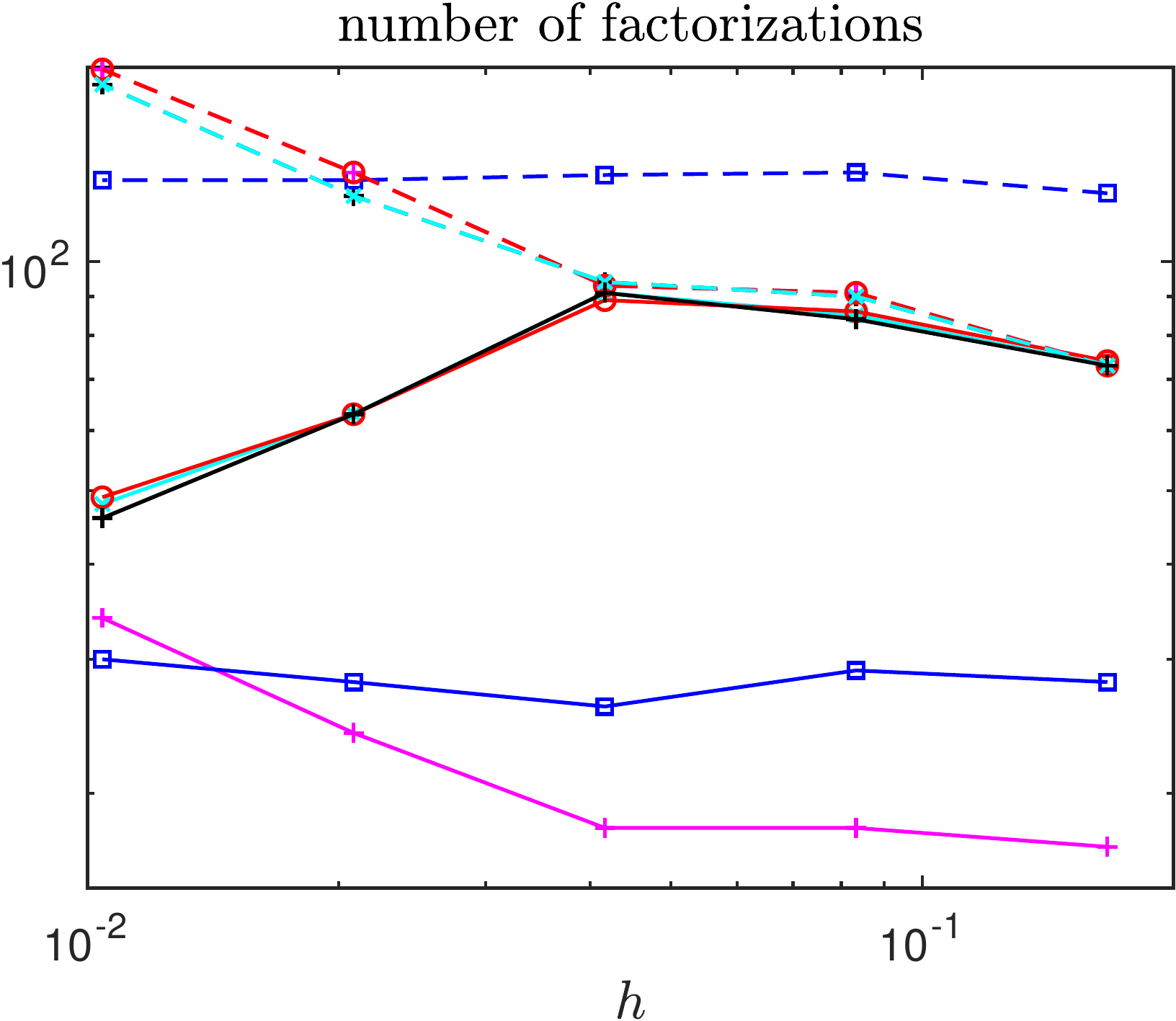}
\end{center}
\caption{Number of steps (left) and~factorizations (right) for results for $TOL=10^{-4}$ for all meshes. Color and symbol codes on the right plot are those of the left one. Discontinuous lines correspond to the semi-implicit method~\eqref{semimp} and continuous ones for the IMEX method~\eqref{fully_discrete}.}
\label{fig:steps-facts4}
\end{figure}

In Fig~\ref{fig:steps-facts4} we can also see that the number of factorizations remains fairly independent of the number of steps, as it was the case in the previous example where tolerances decreased with~$h$.

\subsection{Flow past a cylinder}
\label{Se:Flow}
We now present results for the well-known benchmark problem defined in~\cite{bench}. The domain is given by
$$
\Omega=(0,2.2)\times(0,0.41)/\left\{ (x,y) \mid (x-0.2)^2 + (y-0.2)^2 \le 0.0025\right\}
$$
and the time interval is~$[0,8]$. The velocity is identical on both vertical sides, and is given by
$$
\bu(0,y)=\bu(2.2,y)=\frac{6}{0.41^2}\sin\Bigl(\frac{\pi t}{8}\Bigr)\left(\begin{array}{c} y(0.41-y)\\ 0\end{array}\right).
$$
In the rest of the boundary the velocity is set $\bu ={\bf 0}$. At~$t=0$, the initial velocity is ${\bu}={\bf 0}$. The viscosity is set to~$\nu=10^{-3}$
and the forcing term is $\bff={\bf 0}$. 

It is well-known that around $t=4$ a vortex sheet develops behind the cylinder, as it can be seen in Fig.~\ref{speed} where we show
the speed and velocity fields for $t=5$ to~$t=8$. We use the same scale in the four plots of the velocity fields, and the results
we obtain are virtually identical to those in~\cite[Fig.~2]{John_ref}.
\begin{figure}
\begin{center}
\includegraphics[height=5.5truecm]{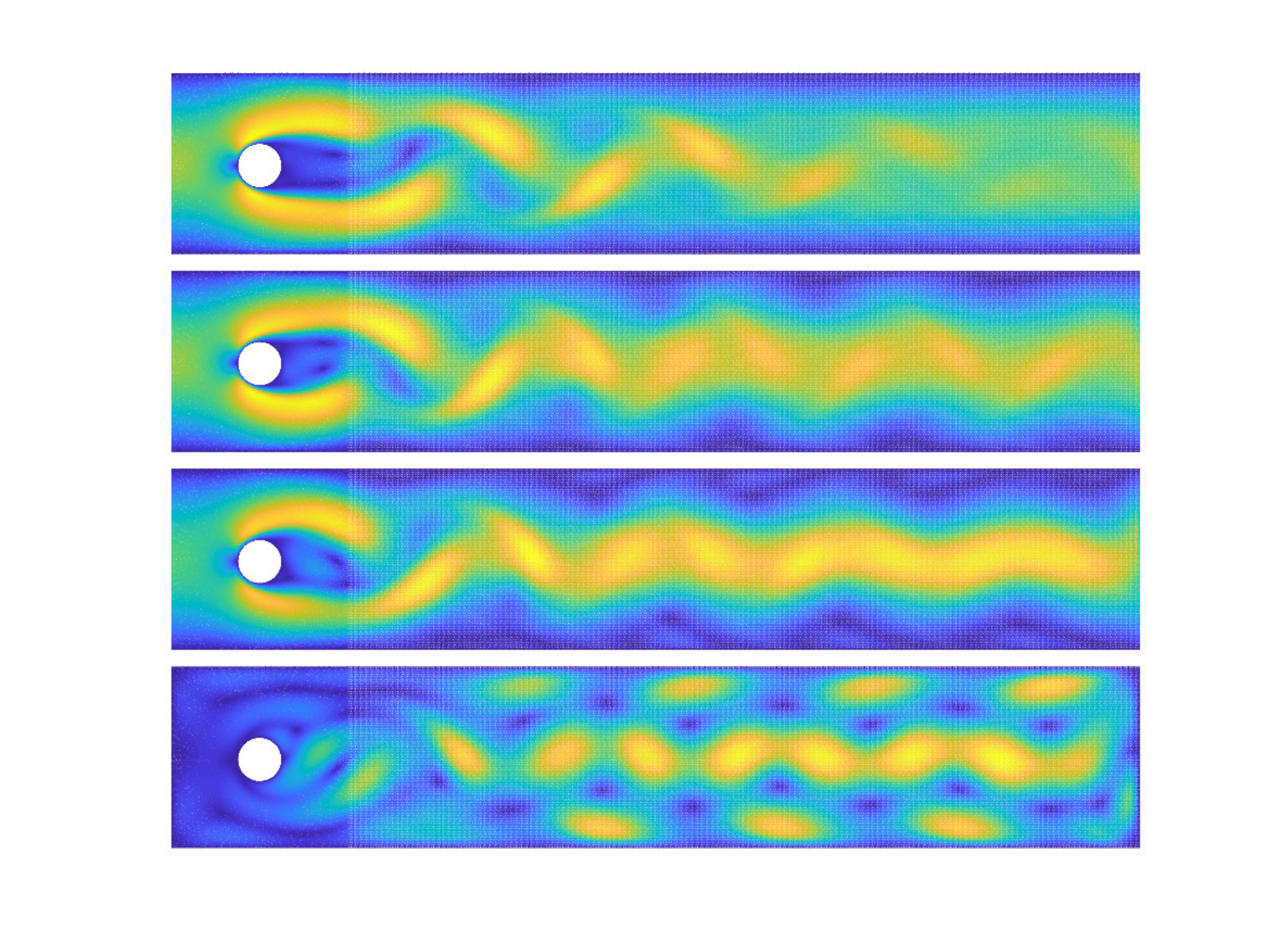}
 \includegraphics[height=5.5truecm]{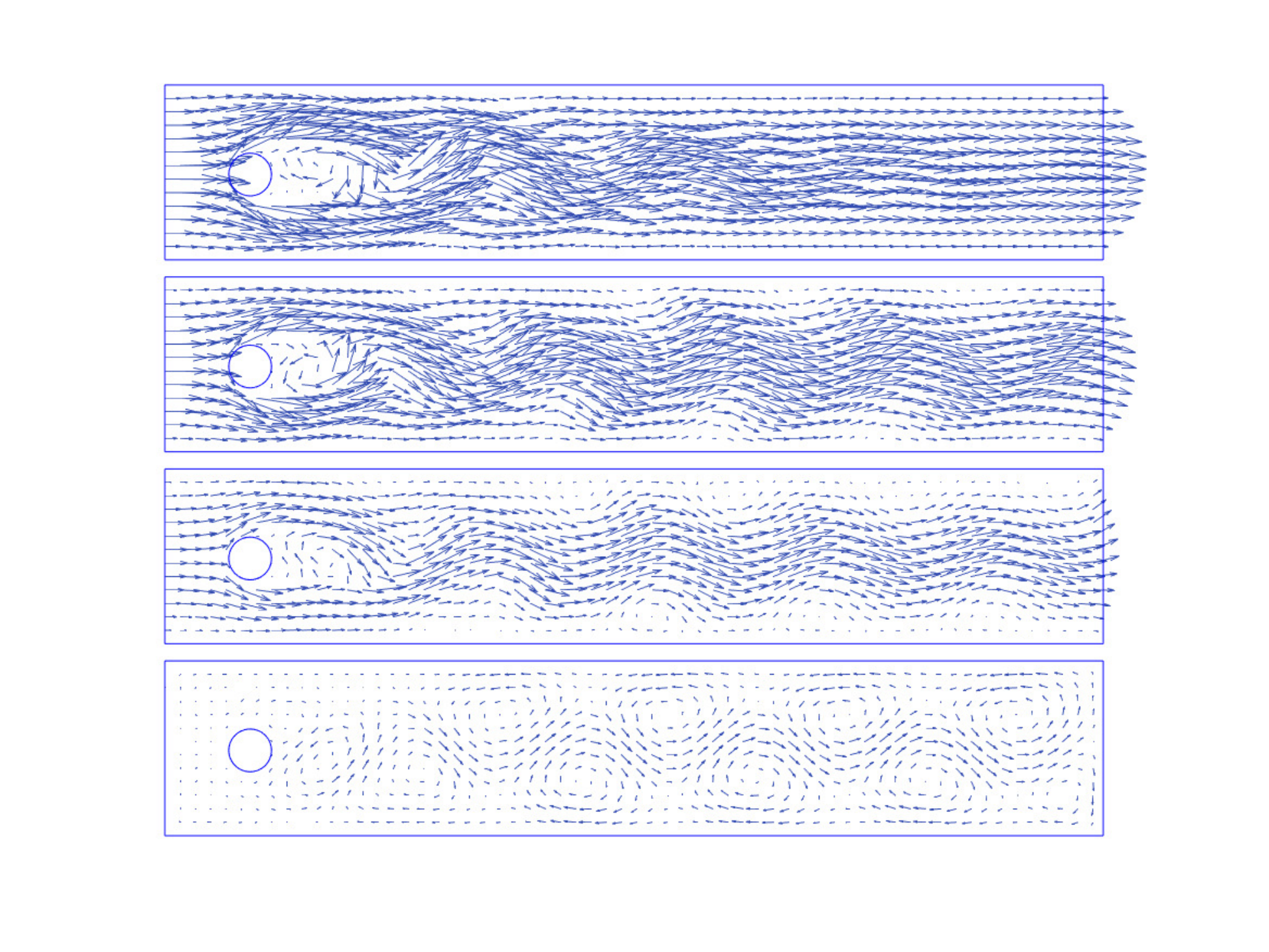}
\caption{Speed contours (left) and velocity field for times $t=5,6,7,8$ (from top to bottom).}
\label{speed}
\end{center}
\end{figure}

In~Fig.~\ref{fig:red} we show the mesh used in all the results below, which has 6624 elements whose diameters range from $5.53\times 10^{-3}$ to~$3.38\times 10^{-2}$. The number of degrees of freedom for the velocity on this mesh is
27168, and for for the pressure, 3480. Following experience in~\cite{proy_gdiv}, the grad-div parameter for the experiments in this section is set to~$\mu=0.01$.

In~Fig~\ref{fig:h_cyl} we show the step sizes used with tolerances $10^{-4}$, $10^{-6}$, and~$10^{-8}$. We notice that, 
for each tolerance, after the very small steps near~$t=0$, step sizes vary by a factor of at least ten, so that, as in the previous section, the algorithm uses variable steps sizes.  It can also be seen that the smaller the tolerance, the smaller the step sizes,
except for the IMEX method~\eqref{fully_discrete} with tolerances ~$10^{-4}$ and~$10^{-6}$ on the interval $[1.7,3.7]$ where step sizes are the same for both tolerances and clearly smaller than those of the semi-implicit method~\eqref{semimp} for the same tolerances. This suggest that, for these two tolerances, the step size in method~\eqref{fully_discrete} is limited by the CFL-type condition~\eqref{CFL2-u} rather than by the roughness of the solution, suggesting, once again, that restriction~\eqref{CFL2-u} is not a consequence of the techniques used in the proof of~Theorem~\ref{th:th2}.

\begin{figure}
\begin{center}
\includegraphics[height=2truecm]{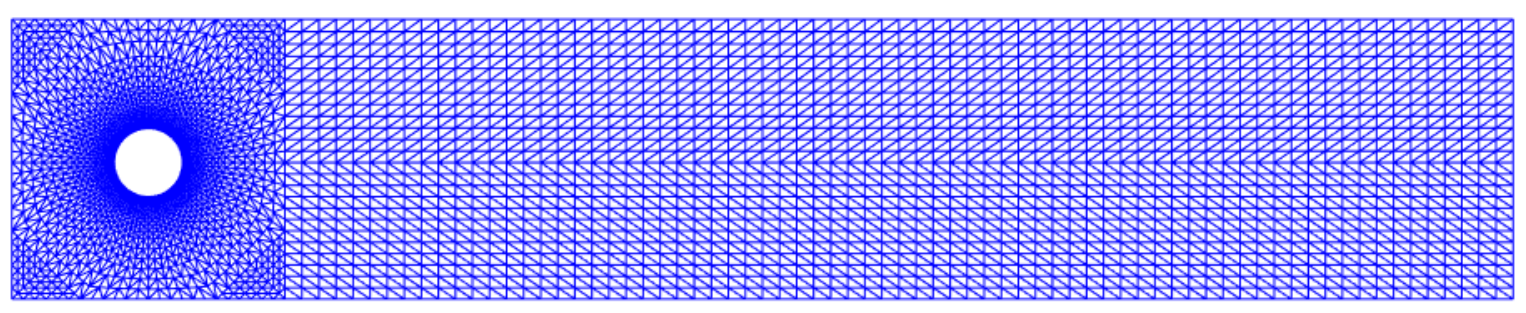}
\caption{The mesh used.}
\label{fig:red}
\end{center}
\end{figure}

\begin{figure}
\begin{center}
\includegraphics[height=3.5truecm]{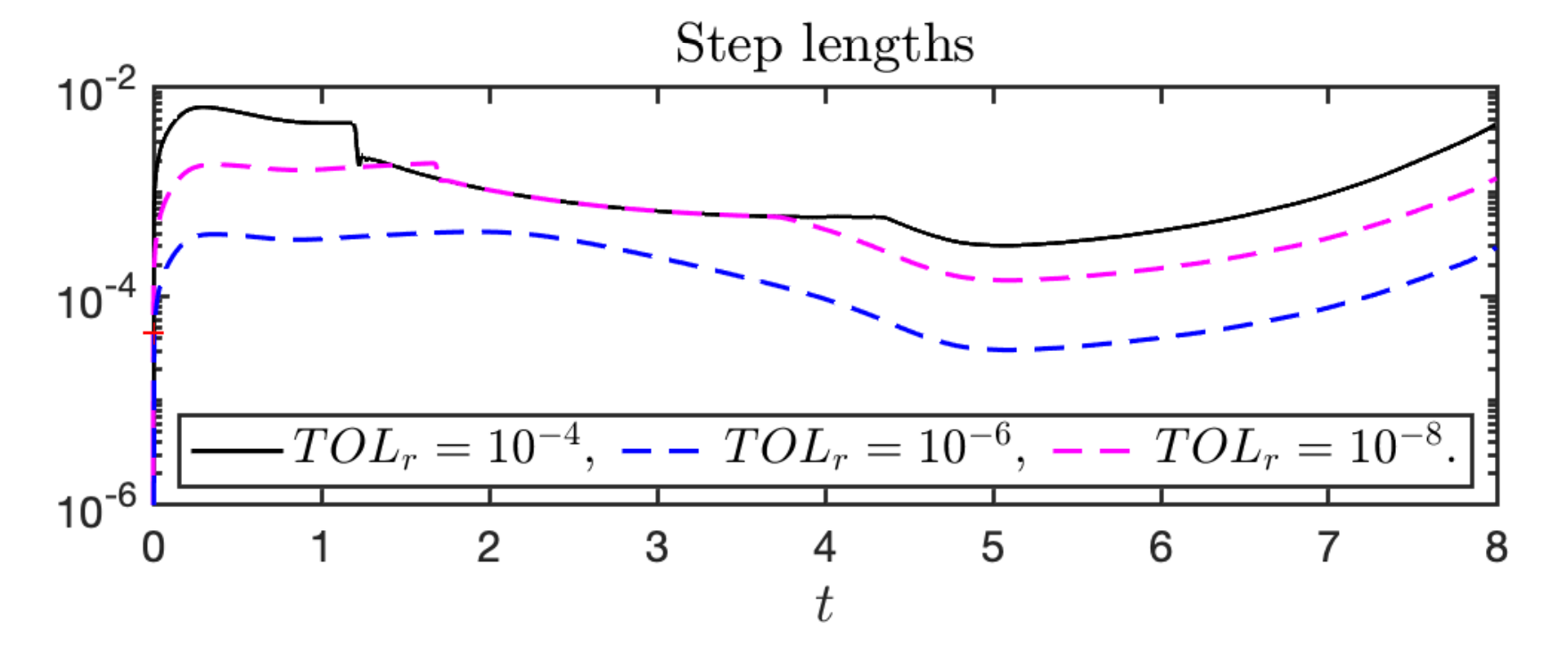}

\includegraphics[height=3.25truecm]{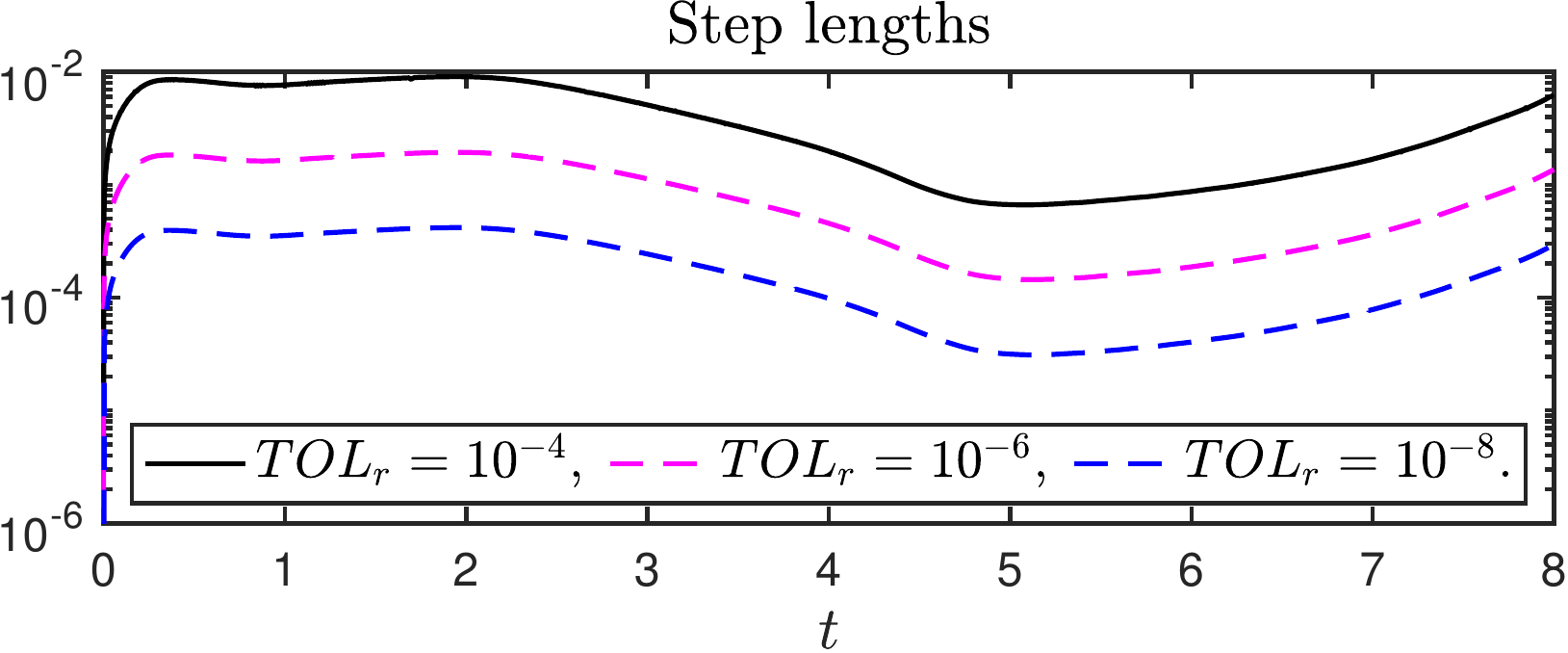}
\end{center}
\caption{Flow around a cylinder: step lengths for $TOLr=10^{-4}$, $10^{-6}$, and~$10^{-8}$. Top: IMEX method~\eqref{fully_discrete} bottom; semi-implicit method~\eqref{semimp}.}
\label{fig:h_cyl}
\end{figure}

\begin{figure}
\begin{center}
\includegraphics[height=5.3truecm]{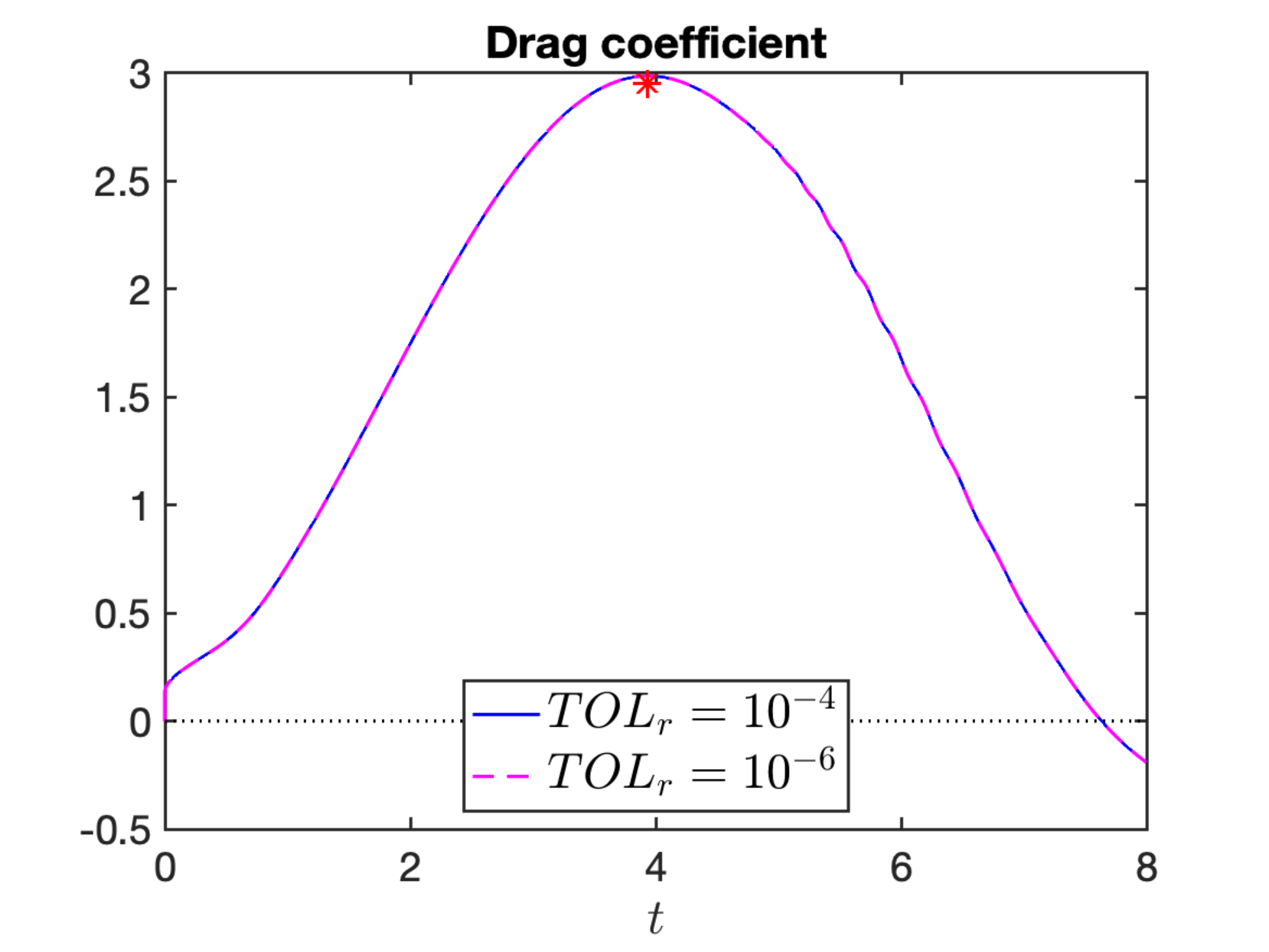} \quad 
\includegraphics[height=5.3truecm]{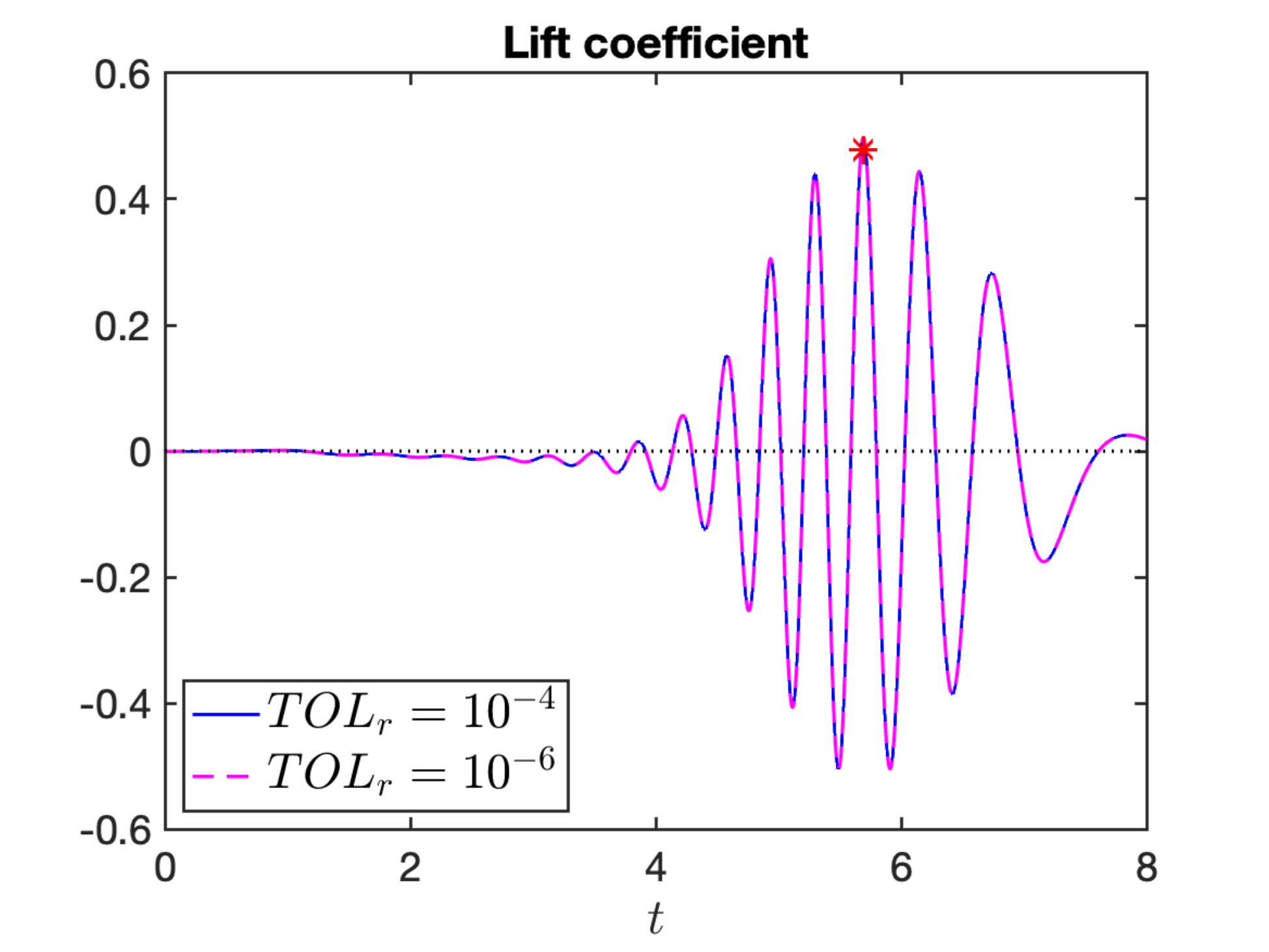}
\end{center}
\caption{Flow around a cylinder: drag and lift coefficients for $TOL_r=10^{-4}$, and~$10^{-6}$. Reference values for the maximum of the lift and drag coefficients marked with an asterisk.}
\label{fig:liftdrag}
\end{figure}

From the practical point of view, we may mention that the number of steps for the three tolerances, $10^{-4}$, $10^{-6}$, and~$10^{-8}$, were
$11221$,       $21187$  and $93671$, respectively, for the IMEX method, and 4384, 19885 and 92086 for the semi-implicit method,  but the number of factorizations were, respectively,
$67$,          $53$, and $92$ for the IMEX method and, 72, 84 and 89 for the semi-implicit method, low figures for both methods, in line with those of the previous section.

Quantities of practical  interest in this example are  the drag and lift coefficients, $c_d$ and~$c_l$, respectively. In Fig.~\ref{fig:liftdrag}, for tolerances $10^{-4}$ and~$10^{-6}$, we show the evolution of the drag and lift coefficients obtained with the IMEX method, which were computed with the following formulae, \cite{John_ref}, \cite{Volker_book}
\begin{align*}
c_d(t_n)&=-20\left( \left(\frac{1}{\Delta t_{n-1}}{\cal D}\bu_h^{n},\bv_d\right)+\nu(\nabla \bu_h^{n},\nabla\bv_d)+b(\bu_h^{n},\bu_h^{n},\bv_d) -(p_h^n(t),\nabla\cdot \bv_d)\right),
\nonumber\\
c_l(t_n)&=-20\left(\left(\frac{1}{\Delta t_{n-1}}{\cal D}\bu_h^{n},\bv_l\right)+ \nu(\nabla\bu_h^n,\nabla\bv_l)+b(\bu_h^{n},\bu_h^{n},\bv_l) -(p_h^n(t),\nabla\cdot \bv_l)\right),
\end{align*}
where $\bv_d$ and~$\bv_l$ are piecewise quadratic functions taking values~$\bv_d=[1,0]^T$ and~$\bv_l=[0,1]^T$ on those nodes on the circumference~$c\equiv (x-0.2)^2 + (y-0.2)^2=0.0025$, and vanishing on triangles without vertices on~$c$.

The results of both tolerances are indistinguishable, suggesting sufficient accuracy already for $TOL_r=10^{-4}$. We also show the reference values of the maximun values, $c_{d,\text{max}}^{\text{ref}}= 2.950921575$
and~$c_{l,\text{max}}^{\text{ref}}=0.47795$, plotted with a red asterisk at their corresponding times
$t_{d,\text{max}}^{\text{ref}}=3.93625$, and~$t_{l,\text{max}}^{\text{ref}}=5.693125$ (all taken from~\cite{John_ref}) showing a very good coincidence. In fact, the relative errors for the times where the maximum values are reached are in both cases below $0.0006$, and for the drag and lift coefficients are $0.012$ and~$0.043$, respectively. Furthermore, there is good agreement between the plots in~Fig.~\ref{fig:liftdrag} and those in~\cite[Fig.~4]{John_ref}. Finally, also of interest is the pressure difference between the front and the back of the cylinder, whose reference value for $t=8$ is, according to~\cite{John_ref}, $\Delta p_8 =-0.1116$. For the tolerance $10^{-4}$, the relative error in $\Delta p_8$ is
below $0.002$.

The results of these two sections suggest that, as predicted by the theory, the IMEX method~(\ref{fully_discrete}) is stable and convergent, but indeed suffers from the CFL-type restriction~\eqref{CFL2-u}.  On the other hand, the computational cost of the semi-implicit 
method~\eqref{semimp} is not 
much larger than the IMEX one, see Figure \ref{fig:steps-facts4}, and has the advantage of having no step restriction.

\section{Conclusions}
The numerical simulation of the Navier-Stokes equations is still a challenge in which there are several questions that deserve some research. Important aspects of the present paper are the following: 
\begin{itemize}	
\item We study the possibility of optimizing the cost of the temporal integration. To this end we consider two IMEX schemes for the incompressible Navier-Stokes equations, which reduce the cost of every time step, compared with an implicit method. In the first method the nonlinear term is explicit and in the second one is semi-implicit. 
\item We prove that using an explicit form for the nonlinear term a CFL type condition is required. This condition is stronger than the usual one in which $(\Delta t) h^{-1}$ has to be bounded (we required $(\Delta t) h^{-2}$ to be bounded). Our numerical experiments confirm  
experimentally that the CFL-type restriction cannot be weakened. 
 \item	Adaptive time stepping is an important tool in Computational Fluid Dynamics. In the present paper, we include the error analysis of a second order backward differentiation formulae (BDF2) scheme with variable time step both for an explicit and semi-implicit form of the nonlinear term. Concerning BDF2 method with variable time step we did not found in the literature any reference with the error analysis for the Navier-Stokes equations. 
\item	Methods for whom robust estimates can be derived enable stable flow simulations for small viscosity coefficients on comparatively coarse grids. Adding grad-div stabilization to the standard mixed finite element formulation we are able to prove error bounds with constants independent of the Reynolds number for the two methods we analized.
\item	In summary, to our knowledge, this is the first time in the literature variable step BDF2 methods are analyzed for the Navier-Stokes equations and also the first time in the literature robust bounds (independent of viscosity coefficient) are proved for Navier-Stokes with fully discrete methods using variable time-step. 
\end{itemize}

\bibliographystyle{abbrv}
\bibliography{references}
\end{document}